\documentclass[12pt]{article}

\usepackage{bm}
\usepackage{xcolor}
\usepackage{amsmath}
\usepackage{amssymb}
\usepackage{graphicx}
\usepackage{amsthm}
\usepackage[round]{natbib}
\usepackage{multirow}
 \usepackage{accents}
\usepackage{epstopdf}

\newtheorem{theorem}{Theorem}[section]

\newtheorem{lemma}{Lemma}[section]
\newtheorem{example}{Example}[section]

\def\S{\mathbb{S}}
\def\R{\mathbb{R}}
\def\N{\mathbb{N}}


\numberwithin{equation}{section}

\title{\bf Optimal designs for regression with spherical data}

\begin{document}

\author{
 {\small Holger Dette, Maria Konstantinou, Kirsten Schorning, Josua G\"osmann} \\
{\small Ruhr-Universit\"at Bochum} \\
{\small Fakult\"at f\"ur Mathematik} \\
{\small 44780 Bochum, Germany}
}
\date{}
\maketitle

\begin{abstract}
In this paper optimal designs  for regression  problems with  spherical  predictors of  arbitrary dimension
are considered. Our work is motivated by applications in material sciences,  where  crystallographic textures
such as the missorientation distribution or  the grain boundary distribution (depending on  a four dimensional 
spherical predictor) are represented by  series of hyperspherical harmonics, which are estimated from experimental or simulated data.  \\
For this type  of estimation problems  we explicitly determine optimal designs with respect to Kiefer's $\Phi_p$-criteria  and 
a  class of orthogonally invariant information criteria recently introduced in the literature.  In particular, we show that  the uniform
distribution on the $m$-dimensional sphere is optimal and construct 
discrete and implementable designs with the same information matrices as the continuous optimal designs. 
Finally, we illustrate the advantages of the new designs for series estimation  by  
hyperspherical harmonics, which are symmetric with respect to the first and second crystallographic point group.
\end{abstract}
\date{}

\maketitle

\section{Introduction}

Regression problems with a predictor of  spherical nature arise in various 
fields such as  geology,  crystallography, astronomy (cosmic microwave background radiation 
data), the  calibration
 of  electromagnetic motion-racking systems
 or  the  representation of spherical viruses  [see \cite{chapchenkim1995}, \cite{zhengetal1995}, \cite{changetal2000}, \cite{schaboog2003}, \cite{genoveseetal2004}, \cite{shintakamurd2007}  among many others] and their 
 parametric and  nonparametric estimation has found
 considerable attention in the  literature. 
 
 Several methods for estimating a spherical regression function nonparameterically have been proposed in the literature.
 \cite{marzpanztay2009,marpantay2014}  investigate kernel type methods, while spherical splines 
 have been considered by  \cite{whaba1981}  and  \cite{alfedetaL1996}.
 A frequently used technique is that of series estimators  based on spherical harmonics [see \cite{abrialetal2008} for example], which - roughly speaking -  generalise  estimators of a regression function on the line based on  Fourier
series  to data on the sphere.  Alternative series estimators have been proposed by   \cite{narcowichetal2006}, \cite{baldietal2009} and \cite{monnier2011} who   suggest   to use spherical wavelets (needlets) in situations where better localisation properties
are required. Most authors consider the $2$-dimensional sphere $\S_3$ in $\mathbb{R}^3 $ as they are interested in 
the development of statistical methodology for concrete applications such as  earth and planetary sciences.

On the other hand, regression models with spherical predictors with a dimension larger than three  have also found considerable attention in the literature, mainly in physics, chemistry and material sciences. Here  
 predictors on the unit sphere
  $$
 \S_m = \{x \in  \R^m : ||x||_2 = 1 \},
 $$ 
 with $m>3$  and  
series expansions in terms of the so called {\it  hyperspherical harmonics} are considered. 
These functions form an orthonormal system with respect to the uniform distribution on the sphere $\S_m $ 
and  
 have been, for example, widely used to solve the
Schroedinger equation by reducing  the problem  to a system of coupled ordinary differential equations  in a single variable [see for example \cite{averywen1982} or  \cite{krivec1998} among many others]. Further applications in this field can be found in \cite{meremianin2009}, who proposed the use of
hyperspherical harmonics for the  representation of  the wave function of the hydrogen atom in the momentum space.
Similarly, \cite{lombardietal2016} suggested to represent the  potential energy surfaces (PES) of atom-molecule or molecular dimers interactions in terms of a   series of
four-dimensional hyperspherical harmonics. Their   method consists in fitting a certain number of points of the PES, previously determined, selected on the basis of geometrical and physical characteristics of the system.  The resulting potential energy function is suitable to serve as a PES for molecular dynamics simulations. \cite{hosseinboretal2013} applied  four-dimensional  hyperspherical harmonics in medical imaging and estimated the coefficients
in the corresponding series expansion via least squares methods to analyse  brain subcortical structures. A further important application of series expansions appears in material sciences, where  crystallographic textures as quaternion distributions are represented  by means of series expansions
based on (symmetrized) hyperspherical harmonics [see \cite{bunge1993}, \cite{zhengetal1995}, \cite{masonschuh2008} and  \cite{mason2009} among many others]. 

It is well known that a carefully designed experiment can improve the statistical inference in regression analysis substantially, and numerous authors have considered the problem of constructing optimal designs for various regression models [see, for example, the monographs of \cite{fedorov1972}, \cite{silvey} and  \cite{pukelsheim}].
On the other hand, despite of its importance, the problem of constructing optimal or efficient designs  for least squares (or alternative) estimation of the coefficients in 
 series expansions based on hyperspherical harmonics has not found much interest in the statistical literature, in particular if the
 dimension $m$ is large. The case $m=2$  corresponding to Fourier regression models  has been discussed intensively [see    \cite{karstu1966}, page 347,   \cite{laustu1985},  \cite{kittittor1988} and
    \cite{detmel2003}
    among many  others]. Furthermore, optimal designs for   series estimators in terms of  spherical harmonics (that is, for $m=3$) have been determined by \cite{detmelpep2005} and \cite{dettwien2009}, however, to the best of our knowledge no results are available for hyperspherical harmonics if the dimension of the predictor  is larger than $3$. 
    
    In the present paper we consider optimal design problems
    for regression models with a spherical predictor of dimension $m >3$
    and explicitly determine optimal designs for series estimators in hyperspherical harmonic expansions.
In Section \ref{sec2} we introduce some basic facts about optimal design theory and hyperspherical harmonics, which will be required for the results presented in this paper. Analytic solutions of the optimal design problem are given in Section \ref{sec31}, where we determine
optimal designs with respect to all  Kiefer's $\Phi_p$-criteria [see \cite{kiefer1974}]  
as well as with respect to a class of optimality criteria recently introduced by \cite{harman2004}. As it turns out the approximate optimal designs are absolute continuous distributions on the sphere and thus cannot be directly implemented in practice. Therefore,  in Section \ref{sec32} we provide discrete designs with the same information matrices as the continuous optimal designs. To achieve this we construct new Gaussian quadrature formulas for integration on the sphere, which are of own interest. In Section \ref{sec4} we investigate the performance of the optimal designs determined in Section \ref{sec32} when they are used in 
typical applications in material sciences. Here  energy functions are represented in terms 
of series of  {\it symmetrized hyperspherical harmonics}  which are obtained as well as defined as linear combinations of the hyperspherical harmonics such that the  symmetry of a  crystallographic point group is reflected in the energy function.
It is demonstrated that the derived designs have very good efficiencies (for the first crystallographic point group the design is in fact 
$D$-optimal). Finally, a proof of a technical result can be found in Appendix \ref{sec5}.\\
The results obtained in this paper provide a first step towards the solution of optimal design problems for regression models with spherical predictors if the dimension is $m>3$ and offer a deeper understanding  of the general mathematical structure of hyperspherical harmonics,  which so far were only  considered in the cases $m=2$ and $m=3$.

\section{Optimal designs and hyperspherical harmonics} \label{sec2}

\subsection{Optimal design theory} \label{sec21}

We consider the linear regression model
\begin{equation} \label{2.1}
E[Y|x] =f^T(x) {\bm c} ~;~ x \in {\cal X} ,
\end{equation}
where $f^T(x) =(f_1(x) , \ldots , f_D (x) )$ is a vector of linearly independent regression functions,
${\bm c}\in \mathbb{R}^D$ is the vector of unknown parameters,
$x$ denotes a real-valued covariate which  varies in a compact design space, say ${\cal X}$ (which will be $\S_m$ in later sections),
and different observations
are assumed to be independent with the same variance, say $\sigma^2 >0$.
 Following \cite{kiefer1974} we define an approximate  design as a probability measure $\xi$ on the set
 ${\cal X}$ (more precisely on its Borel field). If the design $\xi$ has finite support with masses $w_i$ at the points $x_i$ $(i = 1, \dots, k)$ and $n$
observations can be made by the experimenter, this means that the quantities $w_i n$ are rounded to integers, say $n_i$, satisfying
$\sum^k_{i=1} n_i =n$, and the experimenter takes $n_i$ observations at each location $x_i \:  (i=1, \dots, k)$.
The information matrix of the  least squares estimator is defined  by
\begin{equation} \label{2.2}
M(\xi)=  \int_{\cal X}   f(x) f^T(x)  d\xi(x),
\end{equation}
[see \cite{pukelsheim}] and measures the quality of the design $\xi$ as the matrix $\frac {\sigma^2}{n} M^{-1} (\xi)$ can be considered as an approximation of the
covariance matrix $\sigma^2(X^TX)^{-1}$
of the least squares estimator in the corresponding   linear model  ${\bm Y } = X {\bm c} + \varepsilon$.
Similarly,  if the main interest is the estimation of $s$ linear combinations $K^T {\bm c} $, where $K \in \mathbb{R}^{D \times s}$ is a given matrix of rank $s \leq D$, the covariance matrix of the  least squares estimator  for these linear combinations is
given by  $\frac{\sigma^2}{n} (K^T (X^TX)^{-} K)$, where $(X^TX)^{-}$ denotes the generalized inverse of the matrix
$X^TX$ and it is assumed that $\text{range}(K) \subset \text{range}(X^TX)$.
The corresponding analogue of its inverse  for an approximate design $\xi$ satisfying the range inclusion
$\text{range}(K) \subset \text{range}(M(\xi))$
is given by (up to the constant $\frac{\sigma^2}{n}$)
\begin{equation}\label{2.3}
C_K(\xi )=   (K^T M^{-}(\xi) K) ^{-1} ~.
\end{equation}
It follows from \cite{pukelsheim}, Section 8.3, that  for each design $\xi$ there always exists a design $\bar \xi$ with at most $s(s+1)/2$
support points such that $C_K(\xi )= C_K(\bar \xi )$.
An optimal design maximises an appropriate functional of the matrix $M(\xi)$  and numerous criteria have been proposed
in the literature to discriminate between competing designs [see \cite{pukelsheim}]. Throughout this paper we consider
Kiefer's $\Phi_p$-criteria, which are defined   for  $-\infty \leq p < 1$ as
\begin{equation}\label{2.4}
\Phi_p(\xi) = (\text{tr} \big  \{  \big  (C_K(\xi ))^{p}  \big  \} \big  )^{1/p}= (\text{tr}  \big  \{  \big  (K^T M^-(\xi) K)^{-p} \big  \} \big  )^{1/p} ~.
\end{equation}
Following \citet{kiefer1974}, a design $\xi^*$ is called $\Phi_p$-optimal for estimating the linear combinations $K^T {\bm c}$ if
$\xi^*$ maximises the expression $\Phi_p(\xi) $ among all approximate designs $\xi$
for which $K^T {\bm c}$ is estimable, that is, $\text{range}(K) \subset \text{range}(M(\xi))$. This family of optimality criteria includes the well-known criteria of $D$-, $E$- and $A$-optimality corresponding to the cases $p=0$, $p=-\infty$ and $p=-1$, respectively.  \\
Moreover, we consider a generalised version of the criterion of $E$-optimality introduced by \cite{harman2004} 
[see also 
\cite{filhar2011}]. For the information matrix $M(\xi)$ let $\lambda(M(\xi)) = (\lambda_1(M(\xi)), \ldots,\lambda_D(M(\xi)))^T$ be the vector of the eigenvalues of $M(\xi)$ in nondecreasing order. Then, for $s \in 1, \ldots, D$, we define $\Phi_{E_s}(\xi)$ by the sum of the $s$-th smallest eigenvalues of $M(\xi)$, that is,
\begin{equation}\label{e_scrit}
\Phi_{E_s}(\xi)= \sum_{i=1}^s\lambda_i(M(\xi)).
\end{equation} 
For a fixed $s \in \{1, \ldots, D\}$ we call a design $\xi^*$ $\Phi_{E_s}$-optimal if it maximises the term $\Phi_{E_s}(\xi)$ among all approximate designs $\xi$. 

\medskip

In general, the determination of $\Phi_p$-optimal designs and of  $\Phi_{E_s}$-optimal designs in an explicit form is a very difficult task and the corresponding optimal design problems have only been solved in rare circumstances [see for example \cite{cheng1987}, \cite{dettstud1993}, \cite{pukelsheim}, p.241, and \cite{harman2004}]. In the following discussion we will explicitly determine $\Phi_p$-optimal designs
for  regression models which arise from a series expansion of a  function on the $m$-dimensional sphere $\S_m$ in terms of hyperspherical harmonics. It turns out that the $\Phi_p$-optimal designs are also $\Phi_{E_s}$-optimal for an appropriate choice of $s$.\\
We introduce the hyperspherical harmonics next.

\subsection{Hyperspherical harmonics}\label{sec22}

Assume that the design space is given  by the  $m$-dimensional sphere $\S_m = \{ x \in \R^m : ||x||_2=1\}$. The hyperspherical harmonics are functions of $m-1$ dimensionless variables, namely the hyperangles, which describe the points $x=( x_1, \ldots, x_m)^T \in \S_m $
 on the hypersphere by the equations
\begin{eqnarray}\label{polar}
&&x_1 = \cos \theta_1, \\ \nonumber
&&x_2 = \sin \theta_1 \cos \theta_2, \\ \nonumber
&&x_3 = \sin \theta_1 \sin \theta_2 \cos \theta_3, \\ \nonumber
&& \vdots \quad \quad \quad \quad \quad \vdots \\ \nonumber
&&x_{m-1} = \sin \theta_1 \ldots \sin \theta_{m-2} \cos \phi, \\ \nonumber
&&x_{m} = \sin \theta_1 \ldots \sin \theta_{m-2} \sin \phi,
\end{eqnarray}
where $\theta_i \in [0,\pi]$ for all $i=1, \ldots, m-2$, $\phi \in [-\pi, \pi]$ [see, for example, \cite{andrews} or \cite{meremianin2009}].  As noted by \cite{dokmanic}, this choice of coordinates is not unique but rather a matter of convenience since it is a natural generalisation of the spherical polar coordinates in $\mathbb{R}^3$.  \\
In the literature, hyperspherical harmonics are given explicitly in a complex form (see, for example, \cite{vilenkin} and \cite{averywen1982}). Following the notation in \cite{averywen1982}, they are defined as
%\HD{I propose to write $\boldsymbol{\theta_{m-2}} = (\theta_1, \ldots, \theta_{m-2})$! Similarly $\boldsymbol{\mu_{m-3}} $. I started, but... Please change this everywhere!}
\begin{equation*}
\tilde{Y}_{\lambda,\boldsymbol{\mu_{m-3}}, \pm \mu_{m-2}}(\boldsymbol{\theta_{m-2}}, \phi) = \tilde{A}_{\lambda, \boldsymbol{\mu_{m-2}}} \prod_{i=1}^{m-2}{\left[ C_{\mu_{i-1}-\mu_i}^{\mu_i+\frac{m-i-1}{}} (\cos \theta_i ) (\sin \theta_i)^{\mu_i} \right]} e^{\pm i\mu_{m-2}\phi},
\end{equation*}
where $\boldsymbol{\theta_{m-2}} = (\theta_1, \ldots, \theta_{m-2})$,  $\boldsymbol{\mu _{k}} = (\mu_1, \ldots, \mu_{k})$ for
 $k=m-2, m-3$, and
\begin{equation*}
\tilde{A}_{\lambda, \boldsymbol{\mu_{m-2}}} = \frac{1}{\sqrt{2\pi}} \prod_{i=1}^{m-2}{\Big[\frac{2^{2\mu_i + m - i -3}(\mu_{i-1}-\mu_i)! (2\mu_{i-1}+m-i-1) \Gamma^2(\mu_i +\frac{m-i-1}{2})}{\pi(\mu_{i-1}+\mu_i+m-i-2)!}\Big]^{1/2}} ,
\end{equation*}
is a normalising constant, $\lambda:=\mu_0 \geq \mu_1 \geq \mu_2 \geq \ldots \geq \mu_{m-2} \geq 0$ are a set of integers and the functions
$$
C_{\mu_{i-1}-\mu_i}^{\mu_i+\frac{m-i-1}{2}}\big(x \big),$$
 are the Gegenbauer polynomials (of degree $\mu_{i-1}-\mu_i \in \mathbb{N}_0$ with parameter $\mu_i+\frac{m-i-1}{2}$), which are orthogonal with respect to
  the measure
  $$
  (1- x^2)^{\mu_i+(m-i-1)/2-1/2} I_{[-1,1]}(x) dx,
  $$
  (here $I_A(x)$ denotes the indicator function of the set $A$).
 The complex hyperspherical functions are orthogonal to their corresponding complex conjugate and form an orthonormal basis  of the space of  square integrable functions with respect to the uniform distribution on the sphere
 $$
 L^2 (\S_m) = \{ f: \S_m \to \mathbb{C} \mid \int_{\S_m} |f(x)|^2dx < \infty \}.
 $$
 In fact
 the constants $\tilde{A}_{\lambda, \boldsymbol{\mu_{m-2}}}$ are chosen based on this property [see, for example,
  \cite{averywen1982} for more details].

However, as mentioned in \cite{masonschuh2008}, expansions of  real-valued functions on the sphere are easier to handle in terms
of real hyperspherical harmonics which are obtained from the complex hyperspherical harmonics via the linear transformations
\begin{equation}\label{real-def}
\begin{aligned}
Y_{\lambda, \boldsymbol{\mu_{m-3}}, \mu_{m-2}}(\boldsymbol{\theta_{m-2}}, \phi) &=
\frac{(-1)^{\mu_{m-2}}[\tilde{Y}_{\lambda, \boldsymbol{\mu_{m-3}}, \mu_{m-2}}+\tilde{Y}_{\lambda, \boldsymbol{\mu_{m-3}}, -\mu_{m-2}}]}{\sqrt{2}} \\
&= A_{\lambda, \boldsymbol{\mu_{m-3}}} \prod_{i=1}^{m-3}{\Big [ C_{\mu_{i-1}-\mu_i}^{\mu_i+\frac{m-i-1}{2}} (\cos \theta_i ) (\sin \theta_i)^{\mu_i} \Big]} \\
& ~~ \times B_{\mu_{m-3},\mu_{m-2}} P_{\mu_{m-3}}^{\mu_{m-2}}  ( \cos \theta_{m-2}  ) \cos (\mu_{m-2}\phi) ,  \\
Y_{\lambda, \boldsymbol{\mu_{m-3}}, -\mu_{m-2}}(\boldsymbol{\theta_{m-2}}, \phi) & = \frac{(-i)(-1)^{\mu_{m-2}}[\tilde{Y}_{\lambda, \boldsymbol{\mu_{m-3}}, \mu_{m-2}}-\tilde{Y}_{\lambda, \boldsymbol{\mu_{m-3}}, -\mu_{m-2}}]}{\sqrt{2}} \\
& = A_{\lambda, \boldsymbol{\mu_{m-3}}} \prod_{i=1}^{m-3}{\left[ C_{\mu_{i-1}-\mu_i}^{\mu_i+\frac{m-i-1}{2}} (\cos \theta_i ) (\sin \theta_i)^{\mu_i} \right]} \\
& ~~\times  B_{\mu_{m-3},\mu_{m-2}} P_{\mu_{m-3}}^{\mu_{m-2}}  ( \cos \theta_{m-2}  ) \sin (\mu_{m-2}\phi) ,\\
Y_{\lambda, \boldsymbol{\mu_{m-3}}, 0}(\boldsymbol{\theta_{m-2}}, \phi) &= (-1)^{\mu_{m-2}} \tilde{Y}_{\lambda, \boldsymbol{\mu_{m-3}}, 0} \\
& = A_{\lambda, \boldsymbol{\mu_{m-3}}} \prod_{i=1}^{m-3}{\left[ C_{\mu_{i-1}-\mu_i}^{\mu_i+\frac{m-i-1}{2}} (\cos \theta_i ) (\sin \theta_i)^{\mu_i} \right]} \\
&~~\times \frac{B_{\mu_{m-3},0}}{\sqrt{2}} P_{\mu_{m-3}}^{0}  ( \cos \theta_{m-2}  ) ,
\end{aligned}
\end{equation}
where
\begin{equation}\label{real-constant1}
 A_{\lambda, \boldsymbol{\mu_{m-3}}} = \prod_{i=1}^{m-3}{\Big[\frac{2^{2\mu_i + m - i -3}(\mu_{i-1}-\mu_i)! (2\mu_{i-1}+m-i-1) \Gamma^2(\mu_i +\frac{m-i-1}{2})}{\pi(\mu_{i-1}+\mu_i+m-i-2)!}\Big]^{1/2}} , 
\end{equation}
\begin{equation}\label{real-constant2}
B_{\mu_{m-3},\mu_{m-2}} = \Big[ \frac{2(2\mu_{m-3} + 1)(\mu_{m-3}-\mu_{m-2})!}{4 \pi (\mu_{m-3}+\mu_{m-2})!} \Big]^{1/2} ,
\end{equation}
and $P_{\mu_{m-3}}^{\mu_{m-2}}  ( \cos \theta_{m-2}  )$ is the associated Legendre polynomial which can be expressed in terms of a Gegenbauer polynomial via
\begin{equation*}
(-1)^{\mu_{m-2}} \frac{(2 \mu_{m-2})!}{2^{\mu_{m-2}} (\mu_{m-2})!} (\sin \theta_{m-2})^{\mu_{m-2}} C_{\mu_{m-3}-\mu_{m-2}}^{\mu_{m-2}+\frac{1}{2}} ( \cos \theta_{m-2}  ) = P_{\mu_{m-3}}^{\mu_{m-2}} ( \cos \theta_{m-2}  ).
\end{equation*}
 It is easy to check that in the case of $\mathbb{R}^3$, the expressions in \eqref{real-def}, \eqref{real-constant1} and \eqref{real-constant2} give the well known spherical harmonics involving only the associated Legendre polynomial [see Chapter 9 in \cite{andrews} for more details]. \\
The real hyperspherical harmonics defined in \eqref{real-def}, \eqref{real-constant1} and \eqref{real-constant2} preserve the
orthogonality  properties of complex hyperspherical harmonics proven in \cite{averywen1982}. In other words, the real hyperspherical harmonics form an orthonormal basis of the Hilbert space
$$
L^2(\S_m, d\Omega_m) = \Big \{ g: \S_m \to \R ~|~ \int |  g(\boldsymbol{\theta_{m-2}},\phi)|^2 d\Omega_m < \infty \Big \},
$$
 that is,
\begin{equation}\label{orthonormal}
\int Y_{\lambda, \boldsymbol{\mu_{m-3}}, \mu_{m-2}}(\boldsymbol{\theta_{m-2}}, \phi) Y_{\lambda', \boldsymbol{\mu_{m-3}}', \mu_{m-2}'}(\boldsymbol{\theta_{m-2}}, \phi) \,d \Omega_m = \delta_{\lambda \lambda'} \prod_{i=1}^{m-2} \delta_{\mu_i \mu_i'},
\end{equation}
%\HD{Why do we introduce $\lambda$ here?}
where %$\lambda:=\mu_0 \geq \mu_1 \geq \mu_2 \geq \ldots \geq \mu_{m-3} \geq 0$, $-\mu_{m-3} \leq \mu_{m-2} \leq \mu_{m-3}$ and
\begin{equation*}
d \Omega_m = (\sin \theta_1)^{m-2} d\theta_1 (\sin \theta_2)^{m-3} d\theta_2 \ldots (sin \theta_{m-2}) d\theta_{m-2} d\phi,
\end{equation*}
is the element of solid angle. \\

We now consider the linear regression model \eqref{2.1}, where the vector of regression functions is obtained by
a truncated expansion of a function $ g \in  L^2(\S_m, d\Omega_m)$    of order, say, $d$
in terms of hyperspherical harmonics, that is,
\begin{equation*}
\sum_{\lambda=0}^d \sum_{\mu_1=0}^{\lambda} \ldots \sum_{\mu_{m-2}=-\mu_{m-3}}^{\mu_{m-3}} c_{\lambda, \boldsymbol{\mu_{m-3}}, \mu_{m-2}} Y_{\lambda, \boldsymbol{\mu_{m-3}}, \mu_{m-2}} (\boldsymbol{\theta_{m-2}}, \phi).
\end{equation*}
Consequently, we obtain form \eqref{2.1} (using the coordinates $\boldsymbol{\theta_{m-2}}=(\theta_1, \ldots, \theta_{m-2})$, $\phi)$
\begin{equation}\label{reg-model}
E[Y| \boldsymbol{\theta_{m-2}} , \phi] =  f_d^T(\boldsymbol{\theta_{m-2}}, \phi) {\bm c} ,
\end{equation}
where
\begin{eqnarray*}
f_d(\boldsymbol{\theta_{m-2}}, \phi) &=& \big ( Y_{0, 0, \ldots, 0}(\boldsymbol{\theta_{m-2}}, \phi) , Y_{1, 0, \ldots, 0}(\boldsymbol{\theta_{m-2}}, \phi),
Y_{1, 1, 0, \ldots, 0}(\boldsymbol{\theta_{m-2}}, \phi) ,  \\
&& ~~~~~~~~~~~~ \ldots, Y_{1,1,\ldots,1,-1}(\boldsymbol{\theta_{m-2}}, \phi), \ldots, Y_{d,d,\ldots, d} (\boldsymbol{\theta_{m-2}}, \phi)\big )^T ,
\end{eqnarray*}
is the vector of hyperspherical harmonics of order $d$ and the vector of parameters is given by
\begin{equation*}
{\bm c} = \left( c_{0, 0, \ldots, 0}, c_{1, 0, \ldots, 0}, c_{1, 1, 0, \ldots, 0}, \ldots, c_{1,1,\ldots,1,-1}, \ldots, c_{d,d,\ldots, d} \right)^T .
\end{equation*}
Note that the dimension of the vectors $f_d$ and $c$ is
\begin{equation}\label{D}
D:= \sum_{\lambda=0}^d \sum_{\mu_1=0}^\lambda \ldots \sum_{\mu_{m-3}=0}^{\mu_{m-4}} \sum_{\mu_{m-2}=-\mu_{m-3}}^{\mu_{m-3}} 1 \ = \ \sum_{\lambda=0}^d \frac{(m+2\lambda-2)(\lambda+m-3)!}{\lambda!(m-2)!} ,
\end{equation}
where the expression for the sums over the $\mu_i$'s ($i=1,\ldots,m-2$) is obtained from  \cite{averywen1982}.

\section{$\Phi_p$- and $\Phi_{E_s}$- optimal designs for hyperspherical harmonics} \label{sec3}
\subsection{Optimal designs with a Lebesgue density} \label{sec31}

In this section we determine $\Phi_p$-optimal designs for estimating the parameters in a series expansion of a function
defined on the unit sphere $\S_m$. The corresponding regression model is defined by \eqref{reg-model} and  as
mentioned in Section \ref{sec21} a $\Phi_p$-optimal (approximate) design maximises the criterion \eqref{2.4} in the class of all  probability
measures $\xi$ on the set  $[0,\pi]^{m-2} \times [-\pi,\pi]$ satisfying the range inclusion    $\text{range}(K) \subset \text{range}(M(\xi))$, where the information matrix
$M(\xi)$ is given by
\begin{equation*}
M(\xi) = \int_{-\pi}^{\pi} \int_0^{\pi} \ldots \int_0^{\pi} f_d(\boldsymbol{\theta_{m-2}}, \phi) f_d^T(\boldsymbol{\theta_{m-2}}, \phi) \,d\xi(\boldsymbol{\theta_{m-2}}, \phi).
\end{equation*}
We are interested in finding a design that is efficient for the estimation of the Fourier coefficients corresponding to the $s(k)$ hyperspherical harmonics  
$$
Y_{k,0,\ldots,0},Y_{k,1,0,\ldots,0}, \ldots, Y_{k,\ldots,k,-k},\ldots,Y_{k,\ldots,k,k},
$$
 where
\begin{equation}\label{s-def}
s(k) = \frac{(m+2k-2)(k+m-3)!}{k!(m-2)!} ,
\end{equation}
and $k \in \{ 0,\ldots,d \}$ denotes a given level of resolution. To relate this to the definition of the $\Phi_p$-optimality criteria, let $q \in \mathbb{N}_0$, $0\leq k_0 <k_1 < \ldots < k_q \leq d$ and $0_{k,l}$ be the $s(k) \times s(l)$ matrix with all entries equal to $0$. Define the matrix
\begin{equation}\label{K}
K^T = (K_{j,l})_{j=0,\ldots,q}^{l=0,\ldots, d} ,
\end{equation}
where
\begin{equation}\label{K-more}
K_{j,l} =
\begin{cases}
0_{k_j, l} & l \neq k_j \\
I_{s(k_j)} & l=k_j
\end{cases}  ,
\end{equation} 
$I_a$ denotes the $a \times a$ identity and $0_{a,b}$ is an $a \times b$ matrix with all entries equal to $0$. Note that $K \in \mathbb{R}^{D \times s}$ where $D=\sum_{\lambda=0}^d s(\lambda)$
is defined in \eqref{D},  
 and that $K^T {\bm c} \in \mathbb{R}^s$  defines a  vector with
 \begin{equation}
\label{sdeff}
s=\sum_{j=0}^q s(k_j)  =  \sum_{j=0}^q \frac{(m+2k_j-2)(k_j+m-3)!}{k_j! (m-2)!},  
\end{equation}
  components, that is 
\begin{equation}
\left\{ c_{k_j, \mu_1, \ldots, \mu_{m-2}}| k_j \geq \mu_1 \geq \ldots \geq \mu_{m-3}; -\mu_{m-3} \leq \mu_{m-2} \leq \mu_{m-3}; j = 0, \ldots, q  \right\},
\end{equation}
($s \leq D$).
The following theorem shows that  the uniform distribution on the hypersphere is $\Phi_{E_s}$- and  $\Phi_p$-optimal for estimating the parameters $K^T {\bm c}$
(for any $-\infty \leq p < 1$) .
\medskip

\begin{theorem}\label{optimal-fourier}
Let $p\in[-\infty,1)$, $0\leq k_0 <k_1 < \ldots < k_q \leq d$ be given indices and denote by $K \in \mathbb{R}^{D \times s}$ the matrix defined by \eqref{K} and \eqref{K-more}.  Consider the design given by the uniform distribution on the hypersphere, that is,
\begin{eqnarray}\label{opt-design}
\xi^* &=& \xi^*(d\theta_1,\ldots,d\theta_{m-2},d\phi) = \frac{d \Omega_m}{\tilde{\Omega}} \\
\nonumber
&=&
 \frac{1}{\tilde{\Omega}} (\sin \theta_1)^{m-2} d\theta_1 (\sin \theta_2)^{m-3} d\theta_2 \ldots (\sin \theta_{m-2}) d\theta_{m-2} d\phi,
\end{eqnarray}
where $\tilde{\Omega}$ is a normalising constant given by
\begin{equation}\label{constant}
\tilde{\Omega} = \int_{-\pi}^{\pi} \int_0^{\pi} \ldots \int_0^{\pi} (\sin \theta_1)^{m-2} d\theta_1 (\sin \theta_2)^{m-3} d\theta_2 \ldots (sin \theta_{m-2}) d\theta_{m-2} d\phi = \frac{N_m}{(m-2)!!} ,
\end{equation}
$N_m = 2(m-2)!! \pi^{m/2} / \Gamma(m/2)$ and  the double factorial $n!!$ for $n \in\mathbb{N}$  is defined by
$$n!! = \begin{cases} \prod_{k=1}^{\tfrac{n}{2}} (2k) \quad &n \mbox{ is even} \\ 
 \prod_{k=1}^{\tfrac{n+1}{2}} (2k-1) \quad & n \mbox{ is odd}  \end{cases}.$$
 \begin{itemize}
 \item[(i)] The  information matrix of $\xi^*$  is given by $M(\xi^*)=\frac{1}{\tilde{\Omega}} I_{D}$, where $D$ is defined in \eqref{D}.
 \item[(ii)]  The design  $\xi^*$  is $\Phi_p$-optimal  for estimating the linear combination $K^T {\bm c}$ in the regression model \eqref{reg-model}.
  \item[(iii)]  Let  $s = \sum_{i=0}^q s(k_i)$ be the number of considered hyperspherical harmonics where $s(k_i)$ is defined by \eqref{s-def}.
  Then   the design $\xi^*$ defined by \eqref{opt-design} is also $\Phi_{E_s}$-optimal.
  \end{itemize}
\end{theorem}

\begin{proof}  We note that the explicit expression for the normalising constant $\tilde{\Omega}$ in \eqref{constant}
is given in equation (30) in \cite{wenavery1985}.  Let $\xi^*$ denote the design corresponding to the density defined by \eqref{opt-design} and \eqref{constant}. Then due to the orthonormality property of the real hyperspherical harmonics, given in equation \eqref{orthonormal}, it follows that
\begin{equation}\label{optinfo}
M(\xi^*)  = \frac{1}{\tilde{\Omega}} I_{D} ,
\end{equation}
where $\tilde{\Omega}$ is defined in equation \eqref{constant}. This proves part (i) of the Theorem. \\
For a proof of (ii) let $p>-\infty$.  According to the general equivalence theorem in \cite{pukelsheim}, Section 7.20, the measure $\xi^*$ is $\Phi_p$-optimal if and only if the inequality
\begin{equation}\label{get}
f_d^T (\boldsymbol{\theta_{m-2}} , \phi) \tilde{\Omega} K \Big (K^T \tilde{\Omega}K\Big)^{-p-1} K^T \tilde{\Omega} f_d
(\boldsymbol{\theta_{m-2}} , \phi)  \leq \text{tr}\Big\{\Big (K^T \tilde{\Omega}K\Big)^{-p}\Big \} ,
\end{equation}
holds for all $\boldsymbol{\theta_{m-2}}  \in [0,\pi]^{m-2} $  and $\phi \in [-\pi,\pi]$.

From the definition of the matrix $K$ given in equations \eqref{K} and \eqref{K-more} we have that $K^T K = I_s$ where $s=\sum_{j=0}^q s(k_j)$ and $s(k_j)$ is given in \eqref{s-def}. Therefore, condition \eqref{get} reduces to
\begin{equation}\label{get2}
s \geq \sum_{j=0}^q \sum_{\mu_1=0}^{k_j} \ldots \sum_{\mu_{m-3}=0}^{\mu_{m-4}} \sum_{\mu_{m-2}=-\mu_{m-3}}^{\mu_{m-3}} \tilde{\Omega} \big (Y_{k_j, \boldsymbol{\mu_{m-3}}, \mu_{m-2}}(\boldsymbol{\theta_{m-2}} , \phi) \big )^2.
\end{equation}
Now the right-hand side can be simplified observing the sum rule for real hyperspherical harmonics, that is
\begin{equation}\label{sum-rule}
\sum_{\mu_1=0}^\lambda \ldots \sum_{\mu_{m-2}=-\mu_{m-3}}^{\mu_{m-3}} \Big( Y_{\lambda, \boldsymbol{\mu_{m-3}}, \mu_{m-2}}(\boldsymbol{\theta_{m-2}}, \phi) \Big)^2 = \frac{(m+2 \lambda-2)(m-4)!!(\lambda+m-3)!}{N_m \lambda! (m-3)!} ,
\end{equation}
where the constant $N_m$ is given by
\begin{equation}\label{nm}
N_m = \frac{2(m-2)!!\pi^{m/2}}{\Gamma\left( \frac{m}{2} \right)} = \begin{cases}
                        (2\pi)^{m/2} &\quad \text{$m$ is even} \\
                        
                        2(2\pi)^{(m-2)/2} &\quad \text{$m$ is odd}
                    \end{cases}
,
\end{equation}
(see \cite{averywen1982}).
Therefore,
 the right-hand side of \eqref{get2} becomes
\begin{equation*}
\sum_{j=0}^q \frac{N_m}{(m-2)!!} \frac{(m+2k_j-2)(m-4)!!(k_j+m-3)!}{N_m k_j! (m-3)!} = \sum_{j=0}^q \frac{(m+2k_j-2)(k_j+m-3)!}{k_j! (m-2)!} = 
s,
\end{equation*}
where the last equality follows from the definition of $s$ in \eqref{sdeff}.
Consequently, the right-hand side and left-hand side of \eqref{get2} coincide, which proves that the design $\xi^*$ corresponding to the density defined by \eqref{opt-design} and \eqref{constant} is $\Phi_p$-optimal for any $p \in (-\infty,1)$ and any matrix $K$ of the form \eqref{K} and \eqref{K-more}. 
The remaining case $p=-\infty$ follows from Lemma 8.15 in \cite{pukelsheim}, which completes the proof of part (ii).  \\
For a proof of  part (iii)  let $\mbox{diag}({\gamma_1, \ldots , \gamma_D })$ denote a diagonal matrix with entries $\gamma_1, \ldots , \gamma_D $ and
let 
 \begin{equation}
  \partial\Phi_{E_s}(\xi^*) = \Big  \{ \mbox{diag}({\gamma_1, \ldots , \gamma_D })  \in  \mathbb{R}^{D \times D}
  \Big | ~\gamma_1, \ldots, \gamma_D  \in [0, 1], \sum_{k=1}^D\gamma_k = s \big \},
 \end{equation}
denote the subgradient of  $\Phi_{E_s}$. Then it follows from  Theorem 4   of \cite{harman2004},  that the design $\xi^*$ is $\Phi_{E_s}$-optimal
 if and only if there exists a matrix $\Gamma \in \partial\Phi_{E_s}(\xi^*)$ such that the inequality 
\begin{equation}\label{getes}
f_d^T (\boldsymbol{\theta_{m-2}} , \phi) \Gamma f_d
(\boldsymbol{\theta_{m-2}} , \phi)  \leq \sum_{k=1}^s\lambda_k(M(\xi^*)),
\end{equation}
holds for all $\boldsymbol{\theta_{m-2}} \in [0, \pi]^{m-2}$ and $\phi\in [-\pi, \pi]$. \\
We now set  $\Gamma = KK^T$ where $K$ is defined by the equations \eqref{K} and \eqref{K-more}. 
Therefore  $\Gamma$ is a diagonal matrix with entries $0$ or $1$, and  
$$\text{tr}(\Gamma) = \mbox{tr}(KK^T) = \mbox{tr}(K^TK) = \text{tr}(I_s) = s,$$
that is, the matrix $\Gamma$ is contained in the subgradient $\partial\Phi_{E_s}(\xi^*)$. 
Using this matrix  in \eqref{getes} the left-hand side of the inequality reduces to 
$$ f_d^T (\boldsymbol{\theta_{m-2}} , \phi) \Gamma f_d
(\boldsymbol{\theta_{m-2}} , \phi) =
\sum_{j=0}^q \sum_{\mu_1=0}^{k_j} \ldots \sum_{\mu_{m-3}=0}^{\mu_{m-4}} \sum_{\mu_{m-2}=-\mu_{m-3}}^{\mu_{m-3}} \big (Y_{k_j, \boldsymbol{\mu_{m-3}}, \mu_{m-2}}(\boldsymbol{\theta_{m-2}} , \phi) \big )^2,$$ 
and  part (i) yields for the right hand side of the inequality
$$\sum_{k=1}^s\lambda_k(M(\xi^*))= \tfrac{s}{\tilde\Omega},$$
where $\tilde\Omega$ is defined by \eqref{constant}. Consequently, 
the inequality  \eqref{getes} is equivalent to \eqref{get2},  which has been proved in the proof of 
part (ii). This completes the proof of  Theorem \ref{optimal-fourier}.
\end{proof}

\subsection{Discrete $\Phi_p$- and $\Phi_{E_s}$- optimal designs} \label{sec32}

While the result of the previous section provides a very elegant solution to the $\Phi_p$-optimal design problem from a mathematical point of view, the derived designs $\xi^*$
cannot be directly implemented as the optimal probability measure is absolute continuous. In practice, if   $n \in \N$ observations are available  to estimate the parameters in the linear regression  model  \eqref{reg-model}, one has  to specify a number, say $k$,  of  different points
$(\boldsymbol{\theta_{m-2}^1}, \phi^1), \ldots , (\boldsymbol{\theta_{m-2}^k}, \phi^k) \in [0, \pi]^{m-2} \times [-\pi, \pi]$ defining by \eqref{polar} the locations on the sphere where observations should be taken, and relative frequencies $n_j /n$ defining the proportion of observations taken at each point $(\sum_{j=1}^k n_j=n$).
The  maximisation of the function  \eqref{2.4} in the class of all measures of this type
yields a non-linear and non-convex discrete optimisation problem, which is usually intractable. 

Therefore, for the construction of optimal or (at least) efficient designs we proceed as follows.  Due to Caratheodory's theorem [see, for example, \cite{silvey}] there always exists a probability measure $\xi$ on the set  $[0,\pi]^{m-2} \times [-\pi,\pi]$ with at most $D(D+1)/2$  
support points such that the information matrices of $\xi$ and $\xi^*$ coincide, that is,
\begin{equation}\label{condition-discrete}
M(\xi)=M(\xi^*)=\frac{1}{\tilde{\Omega}} I_{D}.
\end{equation}
We now   identify such a design $\xi$ assigning at the points $\{ (\boldsymbol{\theta_{m-2}^j}, \phi^j) \}_{j=1}^k$ the weights  $\{ \boldsymbol{\omega^j} \}_{j=1}^k = \{ (\omega_1^j, \omega_2^j, \ldots, \omega_{m-2}^j, \omega_\phi^j) \}_{j=1}^k$ such that the identity  \eqref{condition-discrete} is satisfied, where we simultaneously
try to keep the number $k$ of support points ``small''. The numbers $n_j$ specifying the numbers of repetitions at the different
experimental conditions in the concrete experiment are finally obtained by rounding the numbers $n\boldsymbol{\omega^j} $ to integers [see, for example,
\cite{pukrie1992}].  We begin with an auxiliary result about Gauss quadrature which is of independent interest and is proven in the appendix.

\begin{lemma}\label{lemma}
Let $a$ be a positive and integrable weight function on the interval $[-1,1]$ with 
$
\tilde{a} = \int_{-1}^1 a(x) \,dx ,
$
and  let $-1 \leq x_1 < x_2 < \ldots < x_r \leq 1$ denote $r\in \mathbb{N} $ points with corresponding
 positive weights $\omega_1, \ldots, \omega_r$ ($\sum_{j=1}^r \omega_j = 1$). 
 Then the points $x_i$ and weights $\omega_i$ generate a quadrature formula of degree $z \geq r$, that is 
\begin{equation}\label{quad}
\int_{-1}^1 a(x) x^{\ell} \,dx = \tilde{a} \sum_{j=1}^r \omega_j x_j^{\ell}, \qquad \ell=0, \ldots, z,
\end{equation}
 if and only if the following two conditions are satisfied:
\begin{enumerate}
\item[(A)] The polynomial $V_r(x) = \prod_{j=1}^r (x-x_j)$ is orthogonal with respect to the weight function $a(x)$ to all polynomials of degree $z-r$, that is,
\begin{equation}\label{cond1}
\int_{-1}^1 V_r(x) a(x) x^{\ell} \,dx =0, \qquad \ell=0, \ldots, z-r.
\end{equation}
\item[(B)] The weights $\omega_j$ are given by
\begin{equation}\label{cond2}
\omega_j = \frac{1}{\tilde{a}} \int_{-1}^1 a(x) \ell_j(x) \,dx  \quad j=1,\ldots,r,
\end{equation}
where $\ell_j(x) = \prod_{k=1,k \neq j}^r \frac{x-x_k}{x_k-x_j}$ denotes the $j$th Lagrange interpolation polynomial with nodes $x_1,\ldots,x_r$.
\end{enumerate}
\end{lemma}

In the following, we use Lemma \ref{lemma} for $z=2d$ and the weight function 
$$a(x)=(1-x^2)^{{(m-i-2)}/{2}}.
$$
Note that the   Gegenbauer polynomials $C_r^{(m-i-1)/2}(x)$ are orthogonal with respect to the
weight function $a(x)=(1-x^2)^{{(m-i-2)}/{2}}$ on the interval $[-1,1]$ [see \cite{andrews}, {p. 302}]. Hence the $r$ roots of $C_r^{(m-i-1)/2}(x)$ have multiplicity $1$, are real and located in the interval $(-1,1)$. As condition \eqref{cond1} is satisfied for  $a(x)=(1-x^2)^{{(m-i-2)}/{2}}$, they define together with the corresponding (positive) weights in \eqref{cond2}   a Gaussian quadrature formula.
Therefore, it follows that for any $r \in \{ d+1, \ldots, 2d \}$ there exists at least one quadrature formula $\{ x_i^j, \omega_i^j \}_{j=1}^r$ for every $i=1, \ldots, m-2$, such that \eqref{quad} holds with $a(x)=(1-x^2)^{{(m-i-2)}/{2}}$.
We consider quadrature formulas of this type and define the designs
\begin{equation}\label{design-theta}
\zeta_i = \begin{pmatrix} \theta_1^i & \ldots & \theta_r^i \\ \omega_1^i & \ldots & \omega_r^i \end{pmatrix} ,
\end{equation}
on $[0,\pi]$, where
\begin{equation}\label{discrete-thetas}
\theta_j^i = \arccos x_j^i \qquad i=1, \ldots, m-2;\quad j=1, \ldots, r.
\end{equation}
Similarly we define for any $t \in \mathbb{N}$ and any $\beta \in (-\frac{t+1}{t} \pi, - \pi]$ a design $\nu=\nu (\beta,t)$ on the interval $[-\pi,\pi]$ by
\begin{equation}\label{nu}
\nu = \nu (\beta,t) = \begin{pmatrix} \phi^1 & \ldots & \phi^t \\ \frac{1}{t} & \ldots & \frac{1}{t} \end{pmatrix} ,
\end{equation}
where the points $\phi^j$ are given by
\begin{equation}\label{discrete-phi}
\phi^j = \beta + \frac{2 \pi j}{t}, \qquad j=1,\ldots,t.
\end{equation}
The following theorem shows that designs of the form
\begin{equation}
\label{discropt}
\zeta_1 \otimes \ldots \otimes \zeta_{m-2} \otimes \nu,
\end{equation}
are $\Phi_p$- as well as $\Phi_{E_s}$-optimal designs. 

\begin{theorem}\label{optimal-discrete-fourier}
Let $p \in [-\infty,1)$, $0\leq k_0 <k_1 < \ldots < k_q \leq d$ and $K$ be a matrix defined by \eqref{K} and \eqref{K-more}. For any $t \geq 2d+1$ and any $r \in \{ d+1, \ldots, 2d \}$, the design $\zeta_1 \otimes \ldots \otimes \zeta_{m-2} \otimes \nu$
defined in  \eqref{discropt} is $\Phi_p$-optimal for estimating the coefficients $K^T {\bm c}$. \\
Moreover, if $s=\sum_{i=0}^q s(k_j)$ is  the number of considered hyperspherical harmonics defined in \eqref{sdeff}, then for any   $t \geq 2d+1$ and any $r \in \{ d+1, \ldots, 2d \}$,
 the design $\zeta_1 \otimes \ldots \otimes \zeta_{m-2} \otimes \nu$ defined in  \eqref{discropt} is $\Phi_{E_s}$-optimal.
\end{theorem}

\begin{proof}
The assertion can be established by showing the identity
\begin{equation}\label{discrete-optimality}
M(\zeta_1 \otimes \ldots \otimes \zeta_{m-2} \otimes \nu) = \frac{1}{\tilde{\Omega}} I_{D} ,
\end{equation}
where the dimension $D$ is defined in \eqref{D}.  Let
\begin{align*}
\psi (\phi) &= (\psi_{-d}(\phi), \psi_{-d+1}(\phi), \ldots, \psi_{d}(\phi))^T \\
& = (\sqrt{2} \sin (d\phi), \dots, \sqrt{2} \sin (\phi), 1, \sqrt{2} \cos (\phi), \ldots, \sqrt{2} \cos (d\phi))^T.
\end{align*}
Then the real hyperspherical harmonics defined in \eqref{real-def}, \eqref{real-constant1} and \eqref{real-constant2} can be rewritten as
\begin{eqnarray*}
Y_{\lambda, \boldsymbol{\mu_{m-3}}, \mu_{m-2}}(\boldsymbol{\theta_{m-2}}, \phi) 
&=& \prod_{i=1}^{m-3} \tilde{\gamma}_{\mu_{i-1},\mu_i} \prod _{i=1}^{m-3} \Big[ C_{\mu_{i-1}-\mu_i}^{\mu_i+\frac{m-i-1}{2}}  ( \cos \theta_i  ) (\sin \theta_i)^{\mu_i} \Big]  \\
&& \times 
 \gamma_{\mu_{m-3},\mu_{m-2}} P_{\mu_{m-3}}^{|\mu_{m-2}|}  ( \cos \theta_{m-2} ) \psi_{\mu_{m-2}}(\phi) ,
\end{eqnarray*}
where the constants $\tilde{\gamma}_{\mu_{i-1},\mu_i} $ and  $\gamma_{\mu_{m-3},\mu_{m-2}}$ are defined by
\begin{equation*}
\tilde{\gamma}_{\mu_{i-1},\mu_i} = \Big[\frac{2^{2\mu_i + m - i -3}(\mu_{i-1}-\mu_i)! (2\mu_{i-1}+m-i-1) \Gamma^2(\mu_i +\frac{m-i-1}{2})}{\pi(\mu_{i-1}+\mu_i+m-i-2)!}\Big]^{1/2} ,
\end{equation*}
and
\begin{equation*}
\gamma_{\mu_{m-3},\mu_{m-2}} = \Big[ \frac{(2\mu_{m-3}+1)(\mu_{m-3}-|\mu_{m-2}|)!}{4 \pi (\mu_{m-3}+|\mu_{m-2}|)!} \Big]^{1/2} .
\end{equation*}

Therefore, the identity \eqref{discrete-optimality} is equivalent to the system of equations
\begin{align*}
& \int Y_{\lambda, \boldsymbol{\mu_{m-3}}, \mu_{m-2}}(\boldsymbol{\theta_{m-2}}, \phi) Y_{\lambda', \boldsymbol{\mu_{m-3}}', \mu_{m-2}'}(\boldsymbol{\theta_{m-2}}, \phi)
\,d ( \zeta_1 \otimes  \ldots   \otimes \zeta_{m-2}) (\boldsymbol{\theta_{m-2}}) \,d \nu(\phi) 
\\
& = \prod_{i=1}^{m-3} \tilde{\gamma}_{\mu_{i-1},\mu_i} \prod_{i=1}^{m-3} \tilde{\gamma}_{\mu_{i-1}',\mu_i'} \gamma_{\mu_{m-3},\mu_{m-2}} \gamma_{\mu_{m-3}',\mu_{m-2}'}  \\ 
& \times \int_{-\pi}^{\pi} \int_{0}^{\pi} \ldots \int_{0}^{\pi}  \prod _{i=1}^{m-3} \Big[ C_{\mu_{i-1}-\mu_i}^{\mu_i+\frac{m-i-1}{2}}  ( \cos \theta_i  ) (\sin \theta_i)^{\mu_i} \Big] 
P_{\mu_{m-3}}^{|\mu_{m-2}|}  ( \cos \theta_{m-2}  ) \psi_{\mu_{m-2}}(\phi) \\
&  \prod _{i=1}^{m-3} \Big[ C_{\mu_{i-1}'-\mu_i'}^{\mu_i'+\frac{m-i-1}{2}}  ( \cos \theta_i  ) (\sin \theta_i)^{\mu_i'} \Big] P_{\mu_{m-3}'}^{|\mu_{m-2}'|}  ( \cos \theta_{m-2}  ) \psi_{\mu_{m-2}'}(\phi) 
\,d \zeta_1(\theta_1) \ldots \,d \zeta_{m-2}(\theta_{m-2}) \,d \nu(\phi) \\
&= \frac{1}{\tilde{\Omega}} \delta_{\lambda \lambda'} \delta_{\mu_1 \mu_1'} \ldots \delta_{\mu_{m-2} \mu_{m-2}'} ,
\end{align*}
where
\begin{align*}
\lambda,\lambda'&= 0, \ldots, d; ~ \mu_1=0,\ldots,\lambda; ~\ldots ~ ; ~\mu_{m-3}=0,\ldots, \mu_{m-4}; ~\mu_{m-2}=-\mu_{m-3}, \ldots, \mu_{m-3};  \\
\mu_1'&=0,\ldots,\lambda'~; \ldots  ~;  ~\mu_{m-3}'=0,\ldots, \mu_{m-4}'; ~\mu_{m-2}'=-\mu_{m-3}', \ldots, \mu_{m-3}'.
\end{align*}
Note that
\begin{align*}
\tilde{\Omega}
&= 2 \pi \prod_{i=1}^{m-2} \int_{0}^{\pi} (\sin \theta_i)^{m-i-1} \,d \theta_i = 2 \pi \prod_{i=1}^{m-2} \int_{-1}^1 (1-x^2)^{\frac{m-i-2}{2}} \,dx ,
\end{align*}
 and that $a(x)= (1-x^2)^{\frac{m-i-2}{2}} $ is the weight function defining each of the quadrature formulas for $i=1, \ldots, m-2$. 

Consequently, by  Fubini's theorem the system above  
is satisfied if the following equations hold
\begin{equation}\label{eq1}
\int_{-\pi}^{\pi} \psi_{\mu_{m-2}}(\phi) \psi_{\mu_{m-2}'}(\phi) \,d \nu(\phi) = \delta_{\mu_{m-2} \mu_{m-2}'} ,
\end{equation}
($\mu_{m-2}, \mu_{m-2}' = -d, \ldots, d$)
\begin{align}\label{eq2}
&\gamma_{\mu_{m-3},\mu_{m-2}} \gamma_{\mu_{m-3}',\mu_{m-2}'} \int_0^{\pi} P_{\mu_{m-3}}^{|\mu_{m-2}|} ( \cos \theta_{m-2} ) P_{\mu_{m-3}'}^{|\mu_{m-2}'|} ( \cos \theta_{m-2}  ) \,d \zeta_{m-2}(\theta_{m-2})  \nonumber \\
&= \frac{1}{2\pi \int_{-1}^1 1 \,dx} \delta_{\mu_{m-3} \mu_{m-3}'} \delta_{\mu_{m-2} \mu_{m-2}'} ,
\end{align}
($\mu_{m-3}, \mu_{m-3}' = 0,\ldots,d; \mu_{m-2}=0,\ldots,\mu_{m-3}; \mu_{m-2}'=0,\ldots,\mu_{m-3}'$) and for each $i=1,\ldots,m-3$
\begin{align}\label{eq3}
&\tilde{\gamma}_{\mu_{i-1},\mu_i} \tilde{\gamma}_{\mu_{i-1}',\mu_i'} \int_0^{\pi} C_{\mu_{i-1}-\mu_i}^{\mu_i+\frac{m-i-1}{2}}  ( \cos \theta_i  ) (\sin \theta_i)^{\mu_i} C_{\mu_{i-1}'-\mu_i'}^{\mu_i'+\frac{m-i-1}{2}}  ( \cos \theta_i  ) (\sin \theta_i)^{\mu_i'} \,d \zeta_i(\theta_i) \nonumber \\
& = \frac{1}{\int_{-1}^1 (1-x^2)^{\frac{m-i-2}{2}} \,dx} \delta_{\mu_{i-1} \mu_{i-1}'} \delta_{\mu_i \mu_i'} ,
\end{align}
($\mu_{i-1}, \mu_{i-1}' = 0,\ldots,d; \mu_{i}=0,\ldots,\mu_{i-1}; \mu_{i}'=0,\ldots,\mu_{i-1}'$).

It is well known [see \cite{pukelsheim}] that equation \eqref{eq1} is satisfied for measures of the form \eqref{nu}. Hence in what follows we can restrict ourselves to the case $\mu_{m-2}=\mu_{m-2}'$.

Now the integrand in equation \eqref{eq2} is a polynomial of degree $\mu_{m-3} + \mu_{m-3}' \leq 2d$. Furthermore, since $\zeta_{m-2}$ corresponds to a quadrature formula for $a(x)=1$ that integrates polynomials of degree $2d$ exactly, we have from Lemma \ref{lemma} for $z=2d$ and $a(x)=1$ that
\begin{align*}
&\int_0^{\pi} P_{\mu_{m-3}}^{|\mu_{m-2}|} ( \cos \theta_{m-2}  ) P_{\mu_{m-3}'}^{|\mu_{m-2}'|} ( \cos \theta_{m-2}  ) \,d \zeta_{m-2}(\theta_{m-2})  \\
& = \sum_{j=1}^r \omega_{m-2}^{j} P_{\mu_{m-3}}^{|\mu_{m-2}|} ( x_{m-2}^{j}  ) P_{\mu_{m-3}'}^{|\mu_{m-2}|} ( x_{m-2}^{j}  ) = \frac{1}{2} \int_{-1}^1  P_{\mu_{m-3}}^{|\mu_{m-2}|} ( x  ) P_{\mu_{m-3}'}^{|\mu_{m-2}|} ( x  ) \,dx .
\end{align*}
From \cite{andrews}  p.457 we have that
\begin{equation*}
\int_{-1}^{1} \left[ P_{\mu_{m-3}}^{|\mu_{m-2}|} ( x  ) \right]^2  \,dx = \frac{2(\mu_{m-3}+|\mu_{m-2}|)!}{(2\mu_{m-3}+1)(\mu_{m-3}-|\mu_{m-2}|)!} = \frac{2}{4 \pi \gamma_{\mu_{m-3},\mu_{m-2}} \gamma_{\mu_{m-3},\mu_{m-2}}}.
\end{equation*}
Therefore,
\begin{equation*}
\int_0^{\pi} P_{\mu_{m-3}}^{|\mu_{m-2}|} ( \cos \theta_{m-2}  ) P_{\mu_{m-3}'}^{|\mu_{m-2}'|} ( \cos \theta_{m-2}  ) \,d \zeta_{m-2}(\theta_{m-2}) = \frac{1}{4\pi} \frac{\delta_{\mu_{m-3} \mu_{m-3}'}}{\gamma_{\mu_{m-3},\mu_{m-2}} \gamma_{\mu_{m-3}',\mu_{m-2}}} ,
\end{equation*}
since associated Legendre polynomials are orthogonal on $[-1,1]$. This implies equation \eqref{eq2} and in what follows we can restrict ourselves to the case $\mu_{m-3}=\mu_{m-3}'$.

For establishing the system of equations \eqref{eq3}, we begin with establishing the equation for $i=m-3$, that is,
\begin{align}\label{eq3-1}
&\tilde{\gamma}_{\mu_{m-4},\mu_{m-3}} \tilde{\gamma}_{\mu_{m-4}',\mu_{m-3}} \int_0^{\pi} C_{\mu_{m-4}-\mu_{m-3}}^{\mu_{m-3}+1}  ( \cos \theta_{m-3}  ) (\sin \theta_{m-3})^{\mu_{m-3}} C_{\mu_{m-4}'-\mu_{m-3}}^{\mu_{m-3}+1}  ( \cos \theta_{m-3}  )  \nonumber \\
& (\sin \theta_{m-3})^{\mu_{m-3}} \,d \zeta_{m-3}(\theta_{m-3}) = \frac{1}{\int_{-1}^1 (1-x^2)^{\frac{1}{2}} \,dx} \delta_{\mu_{m-4} \mu_{m-4}'} .
\end{align}
The integrand is a polynomial of degree $2\mu_{m-3}+\mu_{m-4}-\mu_{m-3}+\mu_{m-4}'-\mu_{m-3} = \mu_{m-4} + \mu_{m-4}' \leq 2d$. Also since $\zeta_{m-3}$ corresponds to a quadrature formula for $a(x)=\sqrt{1-x^2}$ that integrates polynomials of degree $2d$ exactly, it follows from Lemma \ref{lemma} for $z=2d$ and $a(x)=\sqrt{1-x^2}$ that
\begin{align*}
&\int_0^{\pi}  (\sin \theta_{m-3})^{2\mu_{m-3}} C_{\mu_{m-4}-\mu_{m-3}}^{\mu_{m-3}+1}  ( \cos \theta_{m-3}  ) C_{\mu_{m-4}'-\mu_{m-3}}^{\mu_{m-3}+1}  ( \cos \theta_{m-3}  ) \,d \zeta_{m-3}(\theta_{m-3}) \\
&= \sum_{j=1}^r \omega_{m-3}^{j} (1-(x_{m-3}^{j})^2)^{\mu_{m-3}} C_{\mu_{m-4}-\mu_{m-3}}^{\mu_{m-3}+1} ( x_{m-3}^{j}  ) C_{\mu_{m-4}'-\mu_{m-3}}^{\mu_{m-3}+1}  ( x_{m-3}^{j}  ) \\
&= \frac{1}{\int_{-1}^1 \sqrt{1-x^2} \,dx} \int_{-1}^1 \sqrt{1-x^2} (1-x^2)^{\mu_{m-3}} C_{\mu_{m-4}-\mu_{m-3}}^{\mu_{m-3}+1}  ( x  ) C_{\mu_{m-4}'-\mu_{m-3}}^{\mu_{m-3}+1}  ( x  ) \,dx \\
&= \frac{1}{\pi /2}  \int_{-1}^1 (1-x^2)^{\mu_{m-3}+\frac{1}{2}} C_{\mu_{m-4}-\mu_{m-3}}^{\mu_{m-3}+1} ( x  ) C_{\mu_{m-4}'-\mu_{m-3}}^{\mu_{m-3}+1}  ( x  ) \,dx.
\end{align*}
From \cite{andrews}, Corollary 6.8.4, we have that
\begin{align*}
&\int_{-1}^1 (1-x^2)^{\mu_{m-3}+\frac{1}{2}} \left[C_{\mu_{m-4}-\mu_{m-3}}^{\mu_{m-3}+1}  ( x  ) \right]^2 \,dx  \\
&= \frac{\pi}{2} \frac{(\mu_{m-4}+\mu_{m-3}+1)!}{2^{2 \mu_{m-3}} (\mu_{m-3}!)^2 (\mu_{m-4}+1)(\mu_{m-4}-\mu_{m-3})!} = \frac{1}{\tilde{\gamma}_{\mu_{m-4},\mu_{m-3}} \tilde{\gamma}_{\mu_{m-4},\mu_{m-3}}}.
\end{align*}
Therefore,
\begin{align*}
&\int_0^{\pi}  (\sin \theta_{m-3})^{2\mu_{m-3}} C_{\mu_{m-4}-\mu_{m-3}}^{\mu_{m-3}+1}  ( \cos \theta_{m-3}  ) C_{\mu_{m-4}'-\mu_{m-3}}^{\mu_{m-3}+1}  ( \cos \theta_{m-3}  ) \,d \zeta_{m-3}(\theta_{m-3}) \\
& =\frac{2}{\pi} \frac{\delta_{\mu_{m-4} \mu_{m-4}'}}{\tilde{\gamma}_{\mu_{m-4},\mu_{m-3}} \tilde{\gamma}_{\mu_{m-4}',\mu_{m-3}}}
\end{align*}
since Gegenbauer polynomials $C_{\mu_{m-4}-\mu_{m-3}}^{\mu_{m-3}+1}(x)$ are orthogonal with respect to $(1-x^2)^{\mu_{m-3}+1/2}$ on the interval 
$[-1,1]$. This implies \eqref{eq3-1} and in what follows we can restrict ourselves to the case $\mu_{m-4}=\mu_{m-4}'$.

It remains to show that if \eqref{eq3} holds for $i=k+1$, that is, if
\begin{align}\label{eq3-2}
&\tilde{\gamma}_{\mu_{k},\mu_{k+1}} \tilde{\gamma}_{\mu_{k}',\mu_{k+1}'} \int_0^{\pi} C_{\mu_{k}-\mu_{k+1}}^{\mu_{k+1}+\frac{m-k-2}{2}}  ( \cos \theta_{k+1}  ) (\sin \theta_{k+1})^{\mu_{k+1}} C_{\mu_{k}'-\mu_{k+1}'}^{\mu_{k+1}'+\frac{m-k-2}{2}}  ( \cos \theta_{k+1} )  \nonumber \\
& (\sin \theta_{k+1})^{\mu_{k+1}'} \,d \zeta_{k+1}(\theta_{k+1}) = \frac{1}{\int_{-1}^1 (1-x^2)^{\frac{m-k-3}{2}} \,dx} \delta_{\mu_{k} \mu_{k}'} \delta_{\mu_{k+1} \mu_{k+1}'} ,
\end{align}
then \eqref{eq3} holds for $i=k$, that is,
\begin{align}\label{eq3-3}
&\tilde{\gamma}_{\mu_{k-1},\mu_k} \tilde{\gamma}_{\mu_{k-1}',\mu_k'} \int_0^{\pi} C_{\mu_{k-1}-\mu_k}^{\mu_k+\frac{m-k-1}{2}}  ( \cos \theta_k  ) (\sin \theta_k)^{\mu_k} C_{\mu_{k-1}'-\mu_k'}^{\mu_k'+\frac{m-k-1}{2}}  ( \cos \theta_k  ) \nonumber \\
& (\sin \theta_k)^{\mu_k'} \,d \zeta_k(\theta_k)  = \frac{1}{\int_{-1}^1 (1-x^2)^{\frac{m-k-2}{2}} \,dx} \delta_{\mu_{k-1} \mu_{k-1}'} \delta_{\mu_k \mu_k'} .
\end{align}
Note that we use somewhat a ``backward induction step'' since $\mu_k \geq \mu_{k+1}$.

Now since \eqref{eq3-2} holds, for proving \eqref{eq3-3} we can restrict ourselves to the case $\mu_k=\mu_k'$. The integrand in \eqref{eq3-3} is a polynomial of degree $2\mu_k+\mu_{k-1}-\mu_k+\mu_{k-1}'-\mu_k = \mu_{k-1} + \mu_{k-1}' \leq 2d$. Furthermore, since $\zeta_k$ corresponds to a quadrature formula for $a(x)=(1-x^2)^{\frac{m-k-2}{2}}$ that integrates polynomials of degree $2d$ exactly, we have from Lemma \ref{lemma} for $z=2d$ and $a(x)=(1-x^2)^{\frac{m-k-2}{2}}$ that
\begin{align*}
&\int_0^{\pi} (\sin \theta_k)^{2 \mu_k} C_{\mu_{k-1}-\mu_k}^{\mu_k+\frac{m-k-1}{2}}  ( \cos \theta_k )  C_{\mu_{k-1}'-\mu_k}^{\mu_k+\frac{m-k-1}{2}}  ( \cos \theta_k  )  \,d \zeta_k(\theta_k) \\
&= \sum_{j=1}^r \omega_j^k (1-(x_k^k)^2)^{\mu_k} C_{\mu_{k-1}-\mu_k}^{\mu_k+\frac{m-k-1}{2}}  ( x_j^k ) C_{\mu_{k-1}'-\mu_k}^{\mu_k+\frac{m-k-1}{2}}  ( x_j^k  ) \\
&=\frac{1}{\int_{-1}^1 (1-x^2)^{\frac{m-k-2}{2}} \,dx} \int_{-1}^1 (1-x^2)^{\mu_k + \frac{m-k-2}{2}} C_{\mu_{k-1}-\mu_k}^{\mu_k+\frac{m-k-1}{2}}  ( x  ) C_{\mu_{k-1}'-\mu_k}^{\mu_k+\frac{m-k-1}{2}}  ( x ) \,dx  \\
& =  \frac{1}{\int_{-1}^1 (1-x^2)^{\frac{m-k-2}{2}} \,dx} \frac{\delta_{\mu_{k-1} \mu_{k-1}'}}{\tilde{\gamma}_{\mu_{k-1},\mu_k} \tilde{\gamma}_{\mu_{k-1}',\mu_k}} ,
\end{align*}
since Gegenbauer polynomials $C_{\mu_{k-1}-\mu_k}^{\mu_k+\frac{m-k-1}{2}}(x)$ are orthogonal with respect to $(1-x^2)^{\mu_k+\frac{m-k-2}{2}}$ on interval $[-1,1]$
and 
\begin{align*}
&\int_{-1}^1 (1-x^2)^{\mu_k + \frac{m-k-2}{2}} \Big[ C_{\mu_{k-1}-\mu_k}^{\mu_k+\frac{m-k-1}{2}}  ( x ) \Big]^2 \,dx \\
&= \frac{\pi (\mu_{k-1}+\mu_k+m-k-2)!}{2^{2\mu_k+m-k-3}(\mu_{k-1}-\mu_k)!(2\mu_{k-1}+m-k-1)\Gamma^2\left(\mu_k + \frac{m-k-1}{2}\right)} = \frac{1}{\tilde{\gamma}_{\mu_{k-1},\mu_k} \tilde{\gamma}_{\mu_{k-1},\mu_k}} .
\end{align*}
[see again \cite{andrews} Corollary 6.8.4].
This implies \eqref{eq3-3} and by induction the system of equations \eqref{eq3} is established which completes the proof of the theorem.
\end{proof}

\bigskip

\begin{example} \label{ex1}
\rm
To illustrate our approach we consider the dimension $m=4$ and
a series expansion of  order $d=4$.  By Theorem \ref{optimal-discrete-fourier} with 
$r=d+1=5$ we have to consider the weight functions
$$a_1(x) = (1-x^2)^{1/2}\ , \quad a_2(x) = 1 .$$
The  corresponding Gegenbauer polynomials are given by 
$$
C_5^1(x) = 32x^5 -32x^3+ 6x\ , \quad C^{1/2}_5(x) = \frac{1}{8}(63x^5-70x^3 + 15 x)\ ,
$$
and  we obtain  the following discrete optimal design $\zeta^*_1 \otimes \zeta^*_2 \otimes \nu^*$ given by
%$\zeta_1 \otimes \zeta_{2} \otimes \nu$ where
\begin{equation}\label{discopt4}
\begin{split}
\zeta^*_1 &=
\begin{pmatrix}
\tfrac{\pi}{6} &\tfrac{\pi}{3}  & \tfrac{\pi}{2} & \tfrac{2\pi}{3} & \tfrac{5\pi}{6}   \\
\tfrac{1}{12} & \tfrac{1}{4} & \tfrac{1}{3} & \tfrac{1}{4} & \tfrac{1}{12}
\end{pmatrix},
\\
\zeta^*_2 &= \begin{pmatrix}
 \arccos (x_2) & \arccos (x_1) &  \tfrac{\pi}{2} & \arccos (-x_1) &  \arccos (-x_2) \\
 \tfrac{322-13\sqrt{70}}{1800} &\tfrac{322+13\sqrt{70}}{1800} & \tfrac{64}{225} &  \tfrac{322+13\sqrt{70}}{1800} &\tfrac{322-13\sqrt{70}}{1800}
\end{pmatrix} ,
\\
\nu^* &=
\begin{pmatrix}
-\tfrac{7}{9}\pi & -\tfrac{5}{9}\pi & -\tfrac{3}{9}\pi & -\tfrac{1}{9}\pi  & \tfrac{1}{9}\pi & \tfrac{3}{9}\pi  & \tfrac{5}{9}\pi &\tfrac{7}{9}\pi & \pi \\
\tfrac{1}{9} & \tfrac{1}{9} &\tfrac{1}{9} &\tfrac{1}{9} &\tfrac{1}{9} &\tfrac{1}{9} &\tfrac{1}{9} &\tfrac{1}{9} &\tfrac{1}{9}
\end{pmatrix},
\end{split}
\end{equation}
where $x_1= \tfrac{1}{3} \sqrt{\frac{1}{7} (35-2\sqrt{70})}$ and $x_2= \tfrac{1}{3} \sqrt{\tfrac{1}{7} (35+2\sqrt{70})}$. 
By Theorem \ref{optimal-discrete-fourier} this design is $\Phi_p$- and $\Phi_{E_s}$-optimal.
\\
We now compare the optimal design $\zeta^*_1 \otimes \zeta^*_2 \otimes \nu^*$  with  two uniform designs
 $\hat\zeta_1 \otimes \hat\zeta_2 \otimes \hat\nu$ and  $\tilde\zeta_1 \otimes \tilde\zeta_2 \otimes \tilde\nu$, where
the marginal distributions  of these designs are given by 
\begin{equation}\label{equis4}
\begin{split}
\hat\zeta_1 =
\begin{pmatrix}
0 &\tfrac{\pi}{4}  & \tfrac{\pi}{2} & \tfrac{3\pi}{4} & \pi   \\
\tfrac{1}{5} & \tfrac{1}{5} & \tfrac{1}{5} & \tfrac{1}{5} & \tfrac{1}{5}
\end{pmatrix},
\hat\zeta_2 = \begin{pmatrix}
0 &\tfrac{\pi}{4}  & \tfrac{\pi}{2} & \tfrac{3\pi}{4} & \pi   \\
\tfrac{1}{5} & \tfrac{1}{5} & \tfrac{1}{5} & \tfrac{1}{5} & \tfrac{1}{5}
\end{pmatrix} ,
\hat\nu =\nu^*,
\end{split}
\end{equation}
and 
\begin{equation}\label{modif4}
\begin{split}
\tilde{\zeta}_1 &=
\begin{pmatrix}
\tfrac{\pi}{6} &\tfrac{\pi}{3}  & \tfrac{\pi}{2} & \tfrac{2\pi}{3} & \tfrac{5\pi}{6}   \\
\tfrac{1}{5} & \tfrac{1}{5} & \tfrac{1}{5} & \tfrac{1}{5} & \tfrac{1}{5}
\end{pmatrix}, \\
\tilde{\zeta}_2 &= \begin{pmatrix}
 \arccos (x_2) & \arccos (x_1) &  \tfrac{\pi}{2} & \arccos (-x_1) &  \arccos (-x_2) \\
\tfrac{1}{5} & \tfrac{1}{5} & \tfrac{1}{5} & \tfrac{1}{5} & \tfrac{1}{5}
\end{pmatrix} ,
\tilde{\nu}= \nu^*, 
\end{split}
\end{equation}
respectively. Note that the design $\hat\zeta_1 \otimes \hat\zeta_2 \otimes \hat\nu$ defined by \eqref{equis4} corresponds to a uniform distribution on a grid in $[0, \pi ]\times [0, \pi] \times [-\pi, \pi]$, while the design $\tilde\zeta_1 \otimes \tilde\zeta_2 \otimes \tilde\nu$  in  \eqref{modif4} is an equidistant version of the optimal design $\zeta^*_1 \otimes \zeta^*_2 \otimes \nu^*$. 
In particular, it uses the same support points as the optimal design.

To  compare the  uniform designs with the optimal design  $\zeta^*_1 \otimes \zeta^*_2 \otimes \nu^*$ obtained by Theorem \ref{optimal-discrete-fourier} 
we consider the efficiency
$$
\mbox{eff}(\zeta_1 \otimes \zeta_{2} \otimes \nu) = \frac{\Phi(\zeta_1 \otimes \zeta_{2} \otimes \nu)}{\Phi(\zeta^*_1 \otimes \zeta^*_{2} \otimes \nu^*)},
$$
where $\Phi$ is either the $D$-, $E$- or $\Phi_{E_s}$-optimality criterion. 

We focus on the estimation of $K^T{\boldsymbol{c}}$ where we fix $q=1$, $k_0=0$ and $k_1=4$ and
 $K$ is a block matrix of the form \eqref{K} with appropriate blocks given by \eqref{K-more}. For the case of $\Phi_{E_s}$-optimality we set $s= s(k_0) + s(k_1) = 26$.
The $D$-, $E$- and $\Phi_{E_s}$-efficiencies of the designs $\hat\zeta_1 \otimes \hat\zeta_2 \otimes \hat\nu$ and $\tilde\zeta_1 \otimes \tilde\zeta_2 \otimes \tilde\nu$ are presented in Table \ref{tab1}.\\
For the modified optimal design $\tilde\zeta_1 \otimes \tilde\zeta_2 \otimes \tilde\nu$  (with the same support points as the optimal design) 
we observe a good $D$-efficiency, however the  $\Phi_{E_s}$- and the $E$-efficiencies  are substantially smaller
($54.58\%$ and $49.56\%$, repectively). The uniform design $\hat\zeta_1 \otimes \hat\zeta_2 \otimes \hat\nu$ performs worse with respect to the all considered criteria which shows that 
this uniform design is inefficient in applications.
\begin{table}[h!]
\centering
\begin{tabular}{|r|rrr|}
  \hline
 & $D$-efficiency & $E$-efficiency & $\Phi_{E_s}$-efficiency \\
  \hline
$\hat\zeta_1 \otimes \hat\zeta_2 \otimes \hat\nu$  & 40.64 & 3.18 & 11.27 \\
  \hline
 $\tilde\zeta_1 \otimes \tilde\zeta_2 \otimes \tilde\nu$  & 91.05 & 49.56 & 54.58 \\
   \hline
\end{tabular}
\caption{\it  Efficiencies (in \%) for the the uniform designs $\hat\zeta_1 \otimes \hat\zeta_2 \otimes \hat\nu$ and  $\tilde\zeta_1 \otimes \tilde\zeta_2 \otimes \tilde\nu$  defined in  \eqref{equis4} and
 \eqref{modif4}. }
\label{tab1}
\end{table}
\end{example}

\section{Symmetrized hyperspherical harmonics}  \label{sec4}
In the previous example we have already  shown that the use of the optimal designs yields a substantially more accurate statistical inference in series estimation with hyperspherical harmonics. In this section we consider a typical application of
these functions (more precisely of linear combinations of hyperspherical harmonics)  in material sciences 
and   demonstrate some  advantages of the new designs in this context. Due to space limitations we are not able 
to provide the complete background  on the representations of crystallographic texture however, we explain the main ideas and
refer to  \cite{bunge1993}, \cite{masonschuh2008}   and \cite{PATALA20121383} for further explanation. Some helpful background with more details can also be found in the monograph of
\cite{Marinucci2011}.

\begin{example}\label{ex2}
{\rm
We begin with a brief discussion of the case $m=2$ which - although not relevant for applications in  material sciences - is very helpful for understanding the main idea behind the construction of symmetrized hyperspherical harmonics. In this case the Fourier basis
\begin{align*}
\Big \{ \dfrac{1}{\sqrt{2}},\cos(x),\sin(x),\cos(2x), \sin(2 x),\; \ldots   \Big  \} ,
\end{align*}
is a complete orthonormal system in the Hilbert space ${L}^2([0,2\pi))$ with the common inner product 
$
\langle f,g \rangle = \dfrac{1}{\pi}\int_0^{2\pi} f(x)g(x)dx~.
$
The aim is now to construct an orthonormal basis  for the subspace of functions in ${L}^2([0,2\pi))$, which
are invariant with respect to  the rotation group $\{ R_0, R_{\pi/2} , R_{\pi} , R_{3\pi/2} \}$ defined by 
\begin{align}\label{def:rotgroup}
R_a: 
\begin{cases} 
& [0,2\pi) \to [0,2\pi)\\
& x \mapsto x + a \mod (2\pi)
\end{cases}    ~~~(a \in \{0,\pi/2,\pi,3/2\pi\}),
\end{align}
 that is  $f(\cdot) = f(R_{a}^{-1}(\cdot))$ (or equivalently $f(\cdot) = f(R_{a}(\cdot))$)  for all $a \in \{0,\pi/2,\pi,3/2\pi\}$. 
For this purpose consider the trigonometric polynomial
\begin{align}\label{eq:fourierexp}
f(x)
&=  \dfrac{a_0}{\sqrt{2}} + \sum_{k=1}^\infty a_k \cos(kx) + b_k \sin(kx)
= \sum_{k=0}^\infty c_k^T \cdot Y_k(x)~,
\end{align}
where the vectors  $c_k$ and $Y_k(x)$ are defined by
\begin{align*}
c_k = (a_k, b_k)^T
\quad
\text{and}
\quad
Y_k(x) =
\begin{cases} 
(1/\sqrt{2},0)^T\quad\quad\quad\quad\quad\ k=0,\\
(\cos(kx),\sin(kx))^T\quad \text{otherwise}~
\end{cases},
\end{align*}
respectively,  and assume that the function $f$ is invariant with respect to 
the rotation group $\{ R_0, R_{\pi/2} , R_{\pi} , R_{3\pi/2} \}$, that  is, 
\begin{align*}
\sum_{k=0}^\infty c_k^T \cdot Y_k( x)  = f(x)  = 
f(R_a(x))
= \sum_{k=0}^\infty c_k^T \cdot Y_k(R_a(x)) = \sum_{k=0}^\infty c_k^T \cdot D_k(a) \cdot Y_k(x) ,
%= \sum_{k=0}^\infty Y_k^T(x) \cdot D_k^T(a) \cdot c_k  ~,
\end{align*}
where the matrices $D_k$ are defined by
 \begin{align*}
D_k(a) = \begin{pmatrix}
\cos(ka) & -\sin(ka) \\ 
\sin(ka) & \cos(ka)
\end{pmatrix}~,
\end{align*}
where we have used the addition formulas  for the trigonometric functions. 
This means that $f$ is invariant under $R_a$ if and only if $c_k$ is an eigenvector for the eigenvalue $1$ of $D_k(a)^T$.
Because  $\{ R_0, R_{\pi/2} , R_{\pi} , R_{3\pi/2} \}$ is  
generated by $R_{\pi/2}$, it  suffices to consider the case $a= \pi /2$.  It is now easy to see that 
only the matrices $D_{4\ell} (\pi /2)^T =I_2 $ have the  eigenvalue $1$, which is  of multiplicity $2$ with corresponding 
eigenvectors $ c_{4\ell} = (\beta,0)^T$, $ \tilde c_{4\ell} = (0,\gamma)^T$ ($\beta, \gamma \in \mathbb{R}\setminus \{0\}$).
Consequently, a complete orthonormal basis of  the subset of all functions in $L^2([0,2\pi))$, which are invariant with 
respect to the rotation group  $\{ R_0, R_{\pi/2} , R_{\pi} , R_{3\pi/2} \}$, is obtained by choosing $\beta=\gamma=1$, which yields the linear combinations
$\{ c_{4\ell} ^T Y_{4\ell} (x) , \tilde c_{4\ell} ^T Y_{4\ell} (x)  \}_{\ell =0 ,1 ,\ldots } $  given by 
\begin{align*}
\Big\{ \dfrac{1}{\sqrt{2}},\; \cos(4x),\; \sin(4x),\; \cos(8x),\; \sin(8x),\dots \Big\} .
\end{align*} 
}
\end{example}

\bigskip

In applications in material sciences the dimension is $m=4$ and the groups under consideration are much more complicated and  induce crystal symmetries. For example,
\cite{masonschuh2008} define representations of  crystallographic textures as quaternion distributions (this corresponds to the case  $m=4$ in our notation) by  series expansions in terms of hyperspherical harmonics to reflect sample and crystal symmetries such that the resulting expansions are  more efficient. 
For this purpose they  define the symmetrized hyperspherical harmonics as  specific linear combinations of real hyperspherical harmonics which remain invariant under rotations corresponding to the simultaneous application of a crystal symmetry and sample symmetry operation. The exact definition of the 
symmetrized hyperspherical harmonics is complicated and requires sophisticated arguments from representation theory [see Sections 2 - 4 in  \cite{masonschuh2008}], but - in principle - it follows 
essentially  the same arguments as described in Example \ref{ex2}. \\
More precisely, the  groups induced by the  
crystal symmetry, sample symmetry operation  and  the  level of resolution $\lambda$, 
define    $N(\lambda)$ {\it symmetrized hyperspherical harmonics}  of the  form 
\begin{equation}\label{symhyphar}
\accentset{\therefore}{Z}_\lambda^{\eta}(\theta_1, \theta_2, \phi)  = \sum_{\mu_1=0}^{\lambda} \sum_{\mu_2=-\mu_1}^{\mu_1}{\alpha}^{\eta}_{\lambda, \mu_1, \mu_2} Y_{\lambda, \mu_1, \mu_2}(\theta_1, \theta_2, \phi), \quad \eta = 1, \ldots, N(\lambda)\ ,
\end{equation}
where  the coefficients ${\alpha}^{\eta}_{\lambda, \mu_1, \mu_2}$ are well defined 
and  can be determined form the symmetry properties.
A list of the first at least $30$ symmetrized  hyperspherical harmonics polynomials for the different $11$ point groups can be found in 
the online supplement of   \cite{masonschuh2008}.
 If the coefficients are standardized appropriately, the symmetrized hyperspherical harmonics also form a complete
 orthonormal system, that is, 
 $$
  \int_{-\pi} ^\pi
 \int_0^\pi \int_0^\pi 
 \accentset{\therefore}{Z}_\lambda^{\eta}(\theta_1, \theta_2, \phi) 
  \accentset{\therefore}{Z}_{\lambda'}^{\eta'} (\theta_1, \theta_2, \phi)  \sin^2 \theta_1 \sin  \theta_1 d \theta_1d \theta_2 d \phi 
  = \delta_{\lambda\lambda'}  \delta_{\eta\eta'} ~.
 $$
  Moreover,  any square integrable function $g$ that satisfies the same requirement of crystal and sample symmetry can be uniquely represented as a linear combination of these symmetrized hyperspherical harmonics in the form
\begin{equation}\label{symhypexp}
g(\theta_1, \theta_2, \phi) = \sum_{\lambda=0, 2, \ldots, \infty} \sum_{\eta=1}^{N(\lambda)} c_{\lambda, \eta}\accentset{\therefore}{Z}_\lambda^{\eta}(\theta_1, \theta_2, \phi).
\end{equation}
\cite{PATALA20121383} obtained estimates of the missorientation distribution function by fitting experimentally measured missorientation data to a linear combination of
symmetrized hyperspherical harmonics,
while 
\cite{PATALA20133068} used truncated series to obtain estimates of the grain boundary distribution from simulated data
As these experiments are very expensive and the simulations are very time consuming it is of particular importance to obtain 
good designs for the estimation by series of hypershperical harmonics.
 Therefore, we now consider the linear regression model \eqref{2.1} where the vector of regression functions is obtained by the 
truncated expansion of the function $g$ of order $d$, that is,
\begin{equation}\label{truncsymhypexp}
\sum_{\lambda=0, 2, \ldots, d} \sum_{\eta=1}^{N(\lambda)} c_{\lambda, \eta}\accentset{\therefore}{Z}_\lambda^{\eta}(\theta_1, \theta_2, \phi),
\end{equation}
and investigate  the performance of the designs determined in SectionÊ  
\ref{sec3} in models of the form \eqref{truncsymhypexp}.
Due to space restrictions we concentrate on the case   $d=4$ and  on the symmetrized hyperspherical harmonics 
for samples with orthorhombic symmetry and crystal symmetry corresponding to the crystallographic point groups $1$ and $2$.
Similar results for expansions of higher order and different crystallographic point groups can be obtained following along the same lines. \\
For the crystallographic point group $1$ there are $11$ symmetrized hyperspherical harmonics up to order $d=4$ which 
can be obtained from the  online supplement of   \cite{masonschuh2008} and   are given by
\begin{align} \label{crystgroup1}
\begin{split}
&\accentset{\therefore}{Z}_0^1 = Y_{0,0,0} \\
&\accentset{\therefore}{Z}_4^1 = \sqrt{\tfrac{2}{5}}\quad Y_{4,0,0} + \sqrt{\tfrac{7}{20}}\quad Y_{4,4,0} + \sqrt{\tfrac{1}{4}}\quad Y_{4,4,4} \\
&\accentset{\therefore}{Z}_4^2 = \sqrt{\tfrac{2}{5}}\quad Y_{4,1,0} -\sqrt{\tfrac{1}{10}}\quad Y_{4,3,0} -\sqrt{\tfrac{1}{2}}\quad Y_{4,4,-4} \\
&\accentset{\therefore}{Z}_4^3 = \sqrt{\tfrac{2}{5}}\quad Y_{4,1,1} +\sqrt{\tfrac{3}{80}}\quad Y_{4,3,1} -\sqrt{\tfrac{1}{16}}\quad Y_{4,3,3} +\sqrt{\tfrac{7}{16}}\quad Y_{4,4,-1} +\sqrt{\tfrac{1}{16}}\quad Y_{4,4,-3} \\
&\accentset{\therefore}{Z}_4^4 = \sqrt{\tfrac{2}{5}}\quad Y_{4,1,-1} +\sqrt{\tfrac{3}{80}}\quad Y_{4,3,-1} +\sqrt{\tfrac{1}{16}}\quad Y_{4,3,-3} -\sqrt{\tfrac{7}{16}}\quad Y_{4,4,1} +\sqrt{\tfrac{1}{16}}\quad Y_{4,4,3} \\
&\accentset{\therefore}{Z}_4^5 = \sqrt{\tfrac{4}{7}}\quad Y_{4,2,0} + \sqrt{\tfrac{5}{28}}\quad Y_{4,4,0} -\sqrt{\tfrac{1}{4}}\quad Y_{4,4,4} \\
&\accentset{\therefore}{Z}_4^6 = \sqrt{\tfrac{2}{7}}\quad Y_{4,2,1} -\sqrt{\tfrac{5}{16}}\quad Y_{4,3,-1} +\sqrt{\tfrac{3}{16}}\quad Y_{4,3,-3} +\sqrt{\tfrac{3}{112}}\quad Y_{4,4,1} +\sqrt{\tfrac{3}{16}}  Y_{4,4,3} \\
&\accentset{\therefore}{Z}_4^7 = \sqrt{\tfrac{2}{7}}\quad Y_{4,2,-1} +\sqrt{\tfrac{5}{16}}\quad Y_{4,3,1} +\sqrt{\tfrac{3}{16}}\quad Y_{4,3,3} +\sqrt{\tfrac{3}{112}}\quad Y_{4,4,-1} -\sqrt{\tfrac{3}{16}}  Y_{4,4,-3} \\
&\accentset{\therefore}{Z}_4^8 = \sqrt{\tfrac{4}{7}}\quad Y_{4,2,2} - \sqrt{\tfrac{3}{7}}\quad Y_{4,4,2} \\
&\accentset{\therefore}{Z}_4^9 = \sqrt{\tfrac{2}{7}}\quad Y_{4,2,-2} - \sqrt{\tfrac{1}{2}}\quad Y_{4,3,2} - \sqrt{\tfrac{3}{14}}\quad Y_{4,4,-2} \\
&\accentset{\therefore}{Z}_4^{10} = Y_{4,3,-2}.
\end{split}
\end{align}
Note that the  $26$ functions   $(Y_{0, 0, 0},  Y_{4,0,0}, \ldots, Y_{4,3,3}, Y_{4, 4, -4}, \ldots, Y_{4, 4, 4})^T$ 
define $11$  symmetrized hyperspherical harmonics.
Consequently, considering the symmetries of the crystallographic group $1$ the vector of regression functions 
in model \eqref{2.1}  is of the form
\begin{equation}\label{fcryst1}
f_1^T= (\accentset{\therefore}{Z}_0^1, \accentset{\therefore}{Z}_4^1, \ldots, \accentset{\therefore}{Z}_4^{10})^T.
\end{equation}

\begin{figure}[h!]
\begin{center}
\includegraphics[height=0.8\textheight]{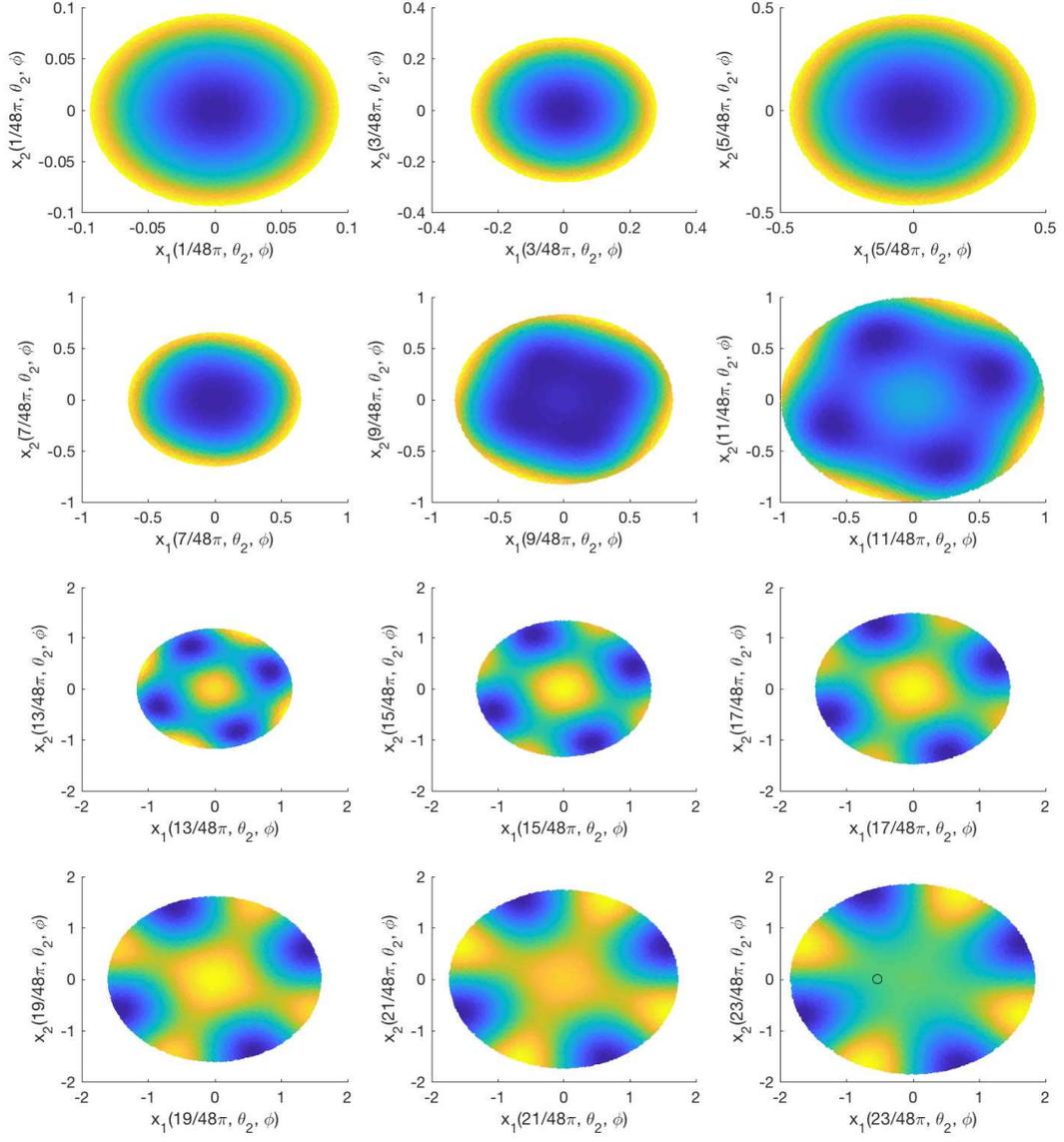}
\end{center}
\caption{\it 
\label{symhyp24}
Visualization of the symmetrized hyperspherical harmonic $\protect\accentset{\therefore}{Z}_4^2$ of the crystallograhic point group $1$ using the projection of the hyperangles onto the two-dimensional disk given by \eqref{eqproject}. For each panel, $\theta_1$ is fixed to a value in $\{\tfrac{1}{48}\pi,\tfrac{3}{48}\pi, \ldots, \tfrac{11}{48}\pi\}$, while $(\theta_2, \phi)$ vary in  $[0, \tfrac{\pi}{2}]\times[-\pi, \pi]$.}
\end{figure}

To illustrate the symmetries induced by the crystallographic group in  the symmetrized hyperspherical harmonics we use a visualization described by \cite{masonschuh2008}. 
For  a fixed hyperangle $(\theta_1, \theta_2, \phi)$, the functional value of $\accentset{\therefore}{Z}_4^j(\theta_1, \theta_2, \phi)$ is presented by using a projection of the hyperangle to an appropriate two-dimensional disk. 
More precisely, we project the hyperangle $(\theta_1, \theta_2, \phi)$ onto a two-dimensional disk by
\begin{equation}\label{eqproject}
P(\theta_1, \theta_2, \phi) = \begin{pmatrix}x_1(\theta_1, \theta_2, \phi) \\ x_2(\theta_1, \theta_2, \phi) \end{pmatrix} = \begin{pmatrix} R(\theta_1, \theta_2) \cos(\phi) \\ R(\theta_1, \theta_2) \sin(\phi) \end{pmatrix} \ , 
\end{equation}
where the function $R(\theta_1, \theta_2)$ is given by 
$$R(\theta_1, \theta_2)  = (3/2)^{1/3}(\theta_1 - \sin(\theta_1)\cos(\theta_1))^{1/3} \sqrt{2(1-|\cos(\theta_2)|)}.$$
For instance the  angle $(\theta_1, \theta_2, \phi) = (\tfrac{11}{48}\pi, \tfrac{\pi}{4}, \pi)$ is projected onto the point
$(x_1, x_2) = (-0.5323, 0)$.  In
Figure \ref{symhyp24} we display  the value (represented by an appropriate  color)
of the  symmetrized harmonic $\accentset{\therefore}{Z}_4^2$  of the crystallographic group $1$ as a function of the coordinates 
 $(x_1(\theta_1, \theta_2, \phi), x_2(\theta_1, \theta_2, \phi))$.  In each of the twelve panels of Figure \ref{symhyp24}, $\theta_1$ is fixed to one of the values $\tfrac{1}{48} \pi,\tfrac{3}{48} \pi, \ldots, \tfrac{11}{48} \pi$, while the angles $\theta_2$ and $\phi$ vary between $[0, \pi/2]$ and $[-\pi, \pi]$, respectively. 
For instance, the value of $\accentset{\therefore}{Z}_4^j(\theta_1, \theta_2, \phi)$ at the   hyperangle $(\theta_1, \theta_2, \phi) = (\tfrac{11}{48}\pi, \tfrac{\pi}{4}, \pi)$ is presented in the bottom right panel in a light green color (see the black circle in the bottom right panel of Figure \ref{symhyp24}).

We now investigate the efficiency of the optimal design for the estimation of the coefficients in the regression model
 \eqref{2.1} with the hyperspherical harmonics up to order $d=4$, that is, the vector of regression functions is given by 
 $$
 f^T= (Y_{0, 0, 0},Y_{1, 0, 0}, \ldots, Y_{1, 1, 1}, \ldots, Y_{4, 4, -4}, \ldots, Y_{4, 4, 4})^T.
 $$
 The optimal design  for this model has been determined in Example \ref{ex1} and  a tedious calculation shows that the design 
 $\zeta_1^*\otimes \zeta_{2}^* \otimes \nu^*$  defined in \eqref{discopt4} satisfies the general equivalence theorem in
  Section 7.20 of  \cite{pukelsheim}.
 Consequently, this design is also $\Phi_p$-optimal in the regression model \eqref{2.1}, where the vector  
 of regression functions is given by the symmetrized hyperspherical harmonics defined in 
  \eqref{fcryst1}, which correspond  to the crystallographic point group $1$.\\
For the crystallographic point group $2$ there are $7$ symmetrized hyperspherical harmonics up to order $d=4$ consisting of a subset of the functions given in \eqref{crystgroup1}. These symmetrized hyperspherical harmonics define a linear  regression model  of the form \eqref{2.1},
where the vector of regression functions $f_2$ is given by 
\begin{equation}\label{fcryst2}
f_2^T = (\accentset{\therefore}{Z}_0^1,\accentset{\therefore}{Z}_4^1,\accentset{\therefore}{Z}_4^2,\accentset{\therefore}{Z}_4^5,\accentset{\therefore}{Z}_4^8,\accentset{\therefore}{Z}_4^9, \accentset{\therefore}{Z}_4^{10})^T.
\end{equation}
In the case of the crystallographic point group $2$ the design $\zeta_1^* \otimes \zeta_{2}^* \otimes \nu^*$ 
defined by  \eqref{discopt4} is not $\Phi_p$-optimal.  However, using particle swarm optimization 
[see \cite{Clerc2006} for details] we determined   the  $D$-efficiency 
of   the design 
$\zeta_1^* \otimes \zeta_{2}^* \otimes \nu^* $ numerically which  is given by
 $81\%$. 
We also investigate the performance of 
  the designs $\hat\zeta_1 \otimes \hat\zeta_2 \otimes \hat\nu$  and  $\tilde\zeta_1 \otimes \tilde\zeta_2 \otimes \tilde\nu$ defined 
  in  equation \eqref{equis4} and \eqref{modif4}   of  Example \ref{ex1}. The  $D$-efficiencies of these two  designs
   are given by  $59.38\%$ and  $74.59\%$, respectively.  Recall  that the latter design uses the same support points as the optimal design.
Our calculations show that  the design $\zeta_1^* \otimes \zeta_{2}^* \otimes \nu^* $  in \eqref{discopt4} 
provides reasonable efficiencies for estimating the coefficients in the regression model \eqref{2.1} with symmetrized hyperspherical harmonics with respect to the 
crystallographic point group $2$, whereas the uniform design $\hat\zeta_1 \otimes \hat\zeta_2 \otimes \hat\nu$ should not be used in this case.

\appendix

\section{A Technical Result}\label{app}
\label{sec5}

\subsection{Proof  of Lemma \ref{lemma}}

Assume that conditions (A) and (B) are satisfied and let $Q(x)$ be an arbitrary polynomial of degree z. The polynomial $Q$ can be represented in the form
\begin{equation*}
Q(x) = P(x)V_r(x) + R(x) ,
\end{equation*}
where $V_r(x)= \prod_{j=1}^r (x-x_j)$ is of degree $r$, the polynomial $P(x)$ is of degree $z-r$ and the polynomial $R(x)$ is of degree less than $r$.
Since $x_1, \ldots, x_r$ are the zeros of $V_r(x)$ we have that $Q(x_j)=R(x_j)$ for all $j=1,\ldots,r$ and furthermore, because the degree of $R(x)$ is at most $r-1$, it can be represented as
\begin{equation*}
R(x) = \sum_{j=1}^r \ell_j(x) R(x_j).
\end{equation*}
Then from conditions (A) and (B) we obtain that
\begin{align*}
\int_{-1}^1 a(x) Q(x) \,dx  & = \int_{-1}^1 a(x) R(x) \,dx = \sum_{j=1}^r R(x_j) \int_{-1}^1 a(x) \ell_j(x) \,dx\\
& = \sum_{j=1}^r R(x_j) \tilde{a} \omega_j = \tilde{a} \sum_{j=1}^r \omega_j Q(x_j) .
\end{align*}
Using $Q(x)=x^{\ell}$, $\ell=0,\ldots,z$ yields the identities in \eqref{quad} and for $\ell=0$ we obtain the expression $\tilde{a} = \int_{-1}^1 a(x) \,dx$.

Now assume that \eqref{quad} is valid. For $\ell=0,\ldots,z-r$ we have that
\begin{equation*}
\int_{-1}^1 V_r(x) a(x) x^{\ell} \,dx = \tilde{a} \sum_{j=1}^r \omega_j V_r(x_j) x_j^{\ell} = 0 ,
\end{equation*}
which gives condition (A). By noting that $\ell_j(x_k) = \delta_{jk}$ we get that
\begin{equation*}
\frac{1}{\tilde{a}} \int_{-1}^1 a(x) \ell_j(x) \,dx  = \sum_{k=1}^r \omega_k \ell_j(x_k) = \sum_{k=1}^r \omega_k \delta_{jk} = \omega_j ,
\end{equation*}
which gives condition (B).

\section*{Acknowledgements}
This work has been supported in part by the
Collaborative Research Center ``Statistical modelling of nonlinear dynamic processes'' (SFB 823, Teilprojekt A1, C2) of the German Research Foundation
(DFG). H. Dette was  partially supported  by a grant from the National Institute of General Medical Sciences of the National
Institutes of Health under Award Number R01GM107639. The content is solely the responsibility of the authors and does not necessarily
represent the official views of the National Institutes of Health.  The authors would like to thank Rebecca Janisch for providing some background information about  crystallographic textures.

\bibliography{hyperspherical_2}
 \bibliographystyle{apalike}

\end{document}